\documentclass{amsproc}
\usepackage[english]{babel}
\usepackage{bm}

\usepackage{stmaryrd}
\usepackage{upgreek}

\usepackage{color}

\usepackage[utf8]{inputenc}
\usepackage{hyperref}
\usepackage{relsize}
\usepackage{amssymb,amsmath}
\usepackage{dsfont}
\usepackage[all]{xy}
\usepackage{mathtools}
\usepackage{mathrsfs}

\renewcommand{\theta}{\uptheta}
\renewcommand{\iota}{\upiota}
\renewcommand{\alpha}{\upalpha}
\renewcommand{\beta}{\upbeta}
\renewcommand{\gamma}{\upgamma}
\renewcommand{\delta}{\updelta}
\renewcommand{\zeta}{\upzeta}
\renewcommand{\pi}{\uppi\hspace{0.05em}}
\renewcommand{\xi}{\upxi}
\renewcommand{\chi}{\upchi}
\renewcommand{\sigma}{\upsigma}
\renewcommand{\Lambda}{\Uplambda}
\renewcommand{\Gamma}{\Upgamma}
\renewcommand{\phi}{\upphi}
\renewcommand{\nu}{\upnu}
\renewcommand{\tau}{\uptau}
\renewcommand{\mu}{\upmu}
\renewcommand{\eta}{\upeta}

\newtheorem{theorem}{Theorem}[section]
\newtheorem{thmx}{Theorem}

\newtheorem{proposition}[theorem]{Proposition}
\newtheorem{lemma}[theorem]{Lemma}

\newtheorem{corollary}[theorem]{Corollary}

\theoremstyle{definition}
\newtheorem{definition}[theorem]{Definition}

\theoremstyle{remark}
\newtheorem{example}[theorem]{Example}
\newtheorem{remark}[theorem]{Remark}

\renewcommand{\AA}{\mathbb{A}}

\newcommand{\CC}{\mathbb{C}}

\DeclareMathOperator{\B}{B\!}
\DeclareMathOperator{\TW}{\mathtt{tw}}

\newcommand{\kac}{\mathtt{a}}
\newcommand{\dd}{\mathbf{d}}

\newcommand{\ee}{\mathbf{e}}

\newcommand{\JH}{\mathtt{JH}}

\newcommand{\FX}{\mathfrak{X}}

\newcommand{\ff}{\mathbf{f}}

\DeclareMathOperator{\Heis}{\mathbf{Heis}}

\DeclareMathOperator{\Walg}{W}
\DeclareMathOperator{\Para}{P}

\newcommand{\Mst}{\mathfrak{M}}

\newcommand{\Msp}{\mathcal{M}}

\newcommand{\DTS}{\mathcal{BPS}}

\DeclareMathOperator{\sst}{-ss}

\DeclareMathOperator{\Hilb}{Hilb}

\DeclareMathOperator{\Hom}{Hom}

\DeclareMathOperator{\Gr}{\mathbf{Gr}}

\DeclareMathOperator{\lmod}{-mod}

\DeclareMathOperator{\BM}{BM}

\newcommand{\fg}{\mathfrak{g}}
\newcommand{\fn}{\mathfrak{n}}
\newcommand{\fm}{\mathfrak{m}}

\DeclareMathOperator{\supp}{supp}

\DeclareMathOperator{\Perv}{\mathbf{Perv}}

\DeclareMathOperator{\Sym}{\mathbf{Sym}}

\DeclareMathOperator{\Spec}{Spec}
\DeclareMathOperator{\Gl}{GL}

\DeclareMathOperator{\UEA}{\mathbf{U}}

\DeclareMathOperator{\id}{id}
\DeclareMathOperator{\Jac}{Jac}

\DeclareMathOperator{\Tr}{Tr}

\DeclareMathOperator{\pt}{pt}
\DeclareMathOperator{\Tot}{Tot}

\DeclareMathOperator{\vir}{vir}

\DeclareMathOperator{\Ho}{\mathcal{H}}
\DeclareMathOperator{\HO}{\mathbf{H}}

\DeclareMathOperator{\Coha}{\mathcal{A}}

\DeclareMathOperator{\TS}{\mathtt{TS}}

\DeclareMathOperator{\BoMo}{BM}

\DeclareMathOperator{\Rees}{\mathbf{R}}

\newcommand{\cdotsh}{\!\cdot\!}

\newcommand{\Dub}{\mathcal{D}}

\newcommand{\Dbc}{\mathcal{D}^b_c}

\newcommand{\cc}{\mathbf{c}}

\newcommand{\BN}{\mathbb{N}}
\newcommand{\BD}{\mathbb{D}}
\newcommand{\BQ}{\mathbb{Q}}
\newcommand{\BC}{\mathbb{C}}
\newcommand{\BZ}{\mathbb{Z}}
\newcommand{\BA}{\mathbb{A}}
\newcommand{\BF}{\mathbb{F}}

\newcommand{\CF}{\mathcal{F}}
\newcommand{\CH}{\mathcal{H}}
\newcommand{\CA}{\mathcal{A}}
\newcommand{\CG}{\mathcal{G}}
\newcommand{\CL}{\mathcal{L}}
\newcommand{\CE}{\mathcal{E}}
\newcommand{\CS}{\mathcal{S}}
\newcommand{\CY}{\mathcal{Y}}
\newcommand{\CN}{\mathcal{N}}

\newcommand{\FP}{\mathfrak{P}}
\newcommand{\FY}{\mathfrak{Y}}
\newcommand{\FZ}{\mathfrak{Z}}

\newcommand{\fp}{\mathfrak{p}}

\newcommand{\phip}[1]{{}^{\mathfrak{p}}\!\phi_{#1}}

\DeclareMathOperator{\Det}{Det}

\DeclareMathOperator{\CoHA}{\mathcal{A}}

\title[Affine BPS algebras, W algebras, and $\CoHA_{\BA^2}^T$]{Affine BPS algebras, W algebras, and the cohomological Hall algebra of $\mathbb{A}^2$}

\author{Ben Davison}
\address{School of Mathematics and Maxwell Institute for Mathematical Sciences\\
University of Edinburgh\\
Edinburgh EH9 3FD\\
UK}
\email{Ben.Davison{\char'100}ed.ac.uk }
\thanks{I was supported by the European Research Council starting grant “Categorified Donaldson--Thomas theory” No. 759967 and a Royal Society University Research Fellowship during the writing of this paper.}
\begin{document}

\maketitle
\vspace{-0.2in}
\begin{center} 
\end{center}
\begin{abstract}
We introduce affinizations and deformations of the BPS Lie algebra associated to a tripled quiver with its canonical cubic potential, and use them to precisely determine the $T$-equivariant cohomological Hall algebra $\CoHA_{\mathbb{A}^2}^T$ of compactly supported coherent sheaves on $\mathbb{A}_{\mathbb{C}}^2$, acted on by a torus $T$.  In particular we show that this algebra is spherically generated for all $T$. 
\end{abstract}

\setcounter{tocdepth}{1}
\tableofcontents

\section{Introduction}
\subsection{Hall algebras}
\label{HAsec}
Let $\Mst_d(\BC[x,y])$ denote the moduli stack of coherent sheaves of length $d$ on $\BA^2\coloneqq\BA^2_{\BC}$.  We also consider the quotient $\Mst^T_d(\BC[x,y])$, by the induced action of various tori $T$ acting linearly on the two coordinates of $\BA^2$.  We consider the Borel--Moore homology of the stack
\[
\CoHA_{\BA^2}^T\coloneqq \bigoplus_{d\in\BZ_{\geq 0}}\HO^{\BM}(\Mst^T_d(\BC[x,y]),\BQ),
\]
which carries the cohomological Hall algebra (CoHA) product defined in \cite{ScVa13}.  This algebra plays a key intermediate role in the geometric representation theory of $\Hilb_d(\BA^2)$ \cite{Nak97,Groj96}, $W$-algebras \cite{ScVa13,Mi07,Arb12,Neg16}, Yangians, and Hall algebras \cite{RSYZ20,RSYZ20b,LY20,GaYa20}.  Setting $T=\{1\}$, the resulting algebra is important in the study of the cohomological Hall algebra of coherent sheaves with zero-dimensional support for an arbitrary smooth surface \cite{KV19}, as well as partially categorifying the degree zero motivic Donaldson--Thomas theory of $\BA^3$, as studied in \cite{BBS}.  In particular, via dimensional reduction \cite{Da13} this algebra is isomorphic to a particular case of the critical cohomological Hall algebra introduced by Kontsevich and Soibelman \cite{KS2}.  
\smallbreak
In this paper we fully describe the algebra $\CoHA_{\BA^2}^T$, for various $T$, via the introduction and study of deformed and affinized BPS Lie algebras $\hat{\fg}_{\tilde{Q},\tilde{W}}^T$ of tripled quivers with their canonical cubic potentials, extending the original BPS Lie algebra $\fg_{\tilde{Q},\tilde{W}}$ of \cite{QEAs} for this class of quivers with potential.
\subsection{Donaldson--Thomas theory}
\label{DT_ssc}
Donaldson--Thomas theory \cite{Thomas1}, is a theory that assigns to a projective Calabi--Yau threefold $X$, a Chern class $\alpha\in \HO(X,\BZ)$, and a stability condition $\zeta$, a number $\omega^{\zeta}_{\alpha}$ ``counting'' $\zeta$-semistable sheaves on $X$ of Chern class $\alpha$.  The inverted commas are due to the fact that the moduli space $\Msp_{\alpha}^{\zeta}(X)$ of such sheaves will generally have strictly positive dimension, i.e. there are infinitely many such sheaves, and so $\omega_{\alpha}^{\zeta}$ is defined by first constructing  a virtual fundamental class of the right expected dimension (zero) and then taking its degree.  
\smallbreak
Subsequently it was realised by Kai Behrend \cite{Behrend} that every quasi-projective complex scheme $Y$ carries a canonical $\BZ$-valued constructible function $\nu_Y$, and the above virtual count can be computed as the Euler characteristic of $Y=\Msp_{\alpha}^{\zeta}(X)$, weighted by $\nu_Y$.  Furthermore, in the event that a scheme $Y$ can be written\footnote{For arbitrary moduli spaces of stable coherent sheaves on Calabi--Yau threefolds, this global critical locus description is too much to ask for, but can still be achieved locally: see \cite{BBD19,BBBD}.} as the critical locus of a function $f$ on a smooth $d$-dimensional ambient variety $U$, its $\nu_Y$-weighted Euler characteristic is equal to the Euler characteristic of the vanishing cycle cohomology $\HO(U,\phip{f}\BQ^{\vir})$, where we define $\BQ^{\vir}=\BQ[d]$, the cohomological shift of the locally constant sheaf on $U$.  Firstly, these facts give us a way to define DT invariants for noncompact moduli spaces, and secondly, they suggest a natural (partial) ``categorification'' of the DT invariant: instead of considering the \textit{number} $\chi(\HO(U,\phip{f}\BQ^{\vir}))=(-1)^d\sum_{i\in\BZ}(-1)^i\dim(\HO^i(U,\phip{f}\BQ))$, we study the cohomology $\HO(U,\phip{f}\BQ^{\vir})$ itself.  Being a vector space, and not a mere number, this cohomology is expected to have (and indeed does have, in cases of interest in this paper) a rich algebraic structure.
\subsection{Noncommutative Donaldson--Thomas theory}
Let $Q$ be a quiver, by which we mean a directed graph, which we always assume to be finite, and which we allow to have multiple edges and loops.  Let $W\in \BC Q/[\BC Q,\BC Q]_{\mathrm{vect}}$ be a \textit{potential}: an element of the quotient of the free path algebra of $Q$ by the vector space spanned by commutators.  Associated to the quiver with potential $(Q,W)$, we may define the Jacobi algebra $\Jac(Q,W)$ as in \S \ref{Quiver_sec} below.  This is a kind of noncommutative threefold, in the sense that the category of finite-dimensional modules over $\Jac(Q,W)$ behaves in many ways like the category of compactly supported coherent sheaves on a 3CY variety\footnote{By which we mean a smooth three-dimensional variety $X$ satisfying $\mathscr{O}_X\cong\omega_X$.}.  We refer to \cite{ginz,MR3338683} for background and details on this point of view.  The moduli stack of finite-dimensional $\Jac(Q,W)$-modules is realised as the critical locus of the function $\Tr(W)$ on the smooth stack of $\BC Q$-modules.  From the discussion in \S \ref{DT_ssc}, it is thus natural to study the vanishing cycle cohomology
\[
\CoHA_{Q,W}\coloneqq \bigoplus_{\dd\in\BN^{Q_0}}\HO(\Mst_{\dd}(Q),\phip{\Tr(W)}\BQ^{\vir}).
\]
This is a categorification of the noncommutative DT theory associated to the pair $(Q,W)$, as initiated in \cite{Sz08} in the case of the noncommutative conifold.  We continue to use the $\vir$ superscript to denote a shift in cohomological degree (see \S \ref{toolkit}).  The cohomology $\CoHA_{Q,W}$ carries an action of $\HO_{\BC^*_u}\coloneqq \HO(\B \BC^*,\BQ)\cong \BQ[u]$ by multiplying by tautological classes, where $u\bullet -$ is a derivation: see \eqref{udef}.  Furthermore, the object $\CoHA_{Q,W}$ carries the cohomological Hall algebra structure introduced by Kontsevich and Soibelman \cite{KS2}.  
\smallbreak
The connection with the algebra $\CoHA_{\BA^2}^T$ arises by considering the quiver $Q^{(3)}$ with one vertex and three loops $a,a^*,\omega$, and considering the potential 
\[
\tilde{W}=\omega[a,a^*].  
\]
Then there is a natural isomorphism of Hall algebras (see \S \ref{dimred_sec})
\[
\CoHA_{Q^{(3)},\tilde{W}}\cong \CoHA_{\BA^2}.
\]
If $T$ scales the $x,y$ directions of $\BA^2$ with weights $w_1,w_2$ respectively, we may consider the $T$-action on $Q^{(3)}$ that scales the arrows $a,a^*,\omega$ with weights $w_1,w_2,-w_1-w_2$ respectively, and there is a similar isomorphism of \textit{deformed} CoHAs $\CoHA_{Q^{(3)},\tilde{W}}^T\cong \CoHA_{\BA^2}^T$.  
\smallbreak
Via a purity result of \cite{preproj}, the graded dimensions of $\CoHA_{Q^{(3)},\tilde{W}}$ encode the refined DT invariants of the Jacobi algebra $\Jac(Q^{(3)},\tilde{W})\cong\BC[x,y,z]$ (equivalently, the DT theory of compactly supported coherent sheaves on $\BA^3$) as defined in \cite{KS1,KS2}, and calculated in \cite{BBS}.  This means that the DT theory of $\Jac(Q^{(3)},\tilde{W})$ enables us to calculate the characteristic function encoding the graded dimensions of $\CoHA_{\BA^2}$ (and from there, using purity again, the characteristic function of $\CoHA^T_{\BA^2}$ --- see Proposition \ref{pur_form}).  More than this, the cohomological DT theory of the category of $\Jac(Q^{(3)},\tilde{W})$-modules, and a thorough study of the deformed affine BPS Lie algebra defined in this paper, enable us to precisely determine the algebra structure on $\CoHA_{\BA^2}^T$.

\subsection{Results}
The algebra $W_{1+\infty}$ is defined to be the universal central extension of the Lie algebra of algebraic differential operators on the complex torus $\mathbb{C}^*$.  A basis for $W_{1+\infty}$ is provided by the set
\[
\{z ^mD^n\;\lvert\;m\in\BZ,n\in \BZ_{\geq 0}\}\cup\{c\}
\]
where we set $D=z(d /d z)$, and $c$ accounts for the central extension, which will not feature in this paper.  We use this basis to consider $W_{1+\infty}$ as a Lie algebra over $\BQ$.  We define 
\begin{equation}
\label{Wplus}
W_{1+\infty}^+=\mathrm{Span}_{\BQ}(z^mD^n\;\lvert \; m\in\BZ_{\geq 1},n\in\BZ_{\geq 0}),
\end{equation}
which is a Lie subalgebra of $W_{1+\infty}$.  The Lie bracket is given by
\begin{equation}
\label{expl_brack}
[z^mD^a,z^nD^b]=z^{m+n}((D+n)^aD^b-D^a(D+m)^b).
\end{equation}
The Lie algebra $W^+_{1+\infty}$ is $\mathbb{Z}_{\geq 0}$-graded, with $n$th graded piece spanned by the operators $z^nD^a$ for $a\geq 0$.  The Lie algebra $W^+_{1+\infty}$ has a $\mathbb{Z}$-filtration given by setting\footnote{There is a doubling of degree in our definition, so that the associated graded object lives only in even degrees, in order to match a later, cohomological grading of Hall algebras.} 
\[
F_iW^+_{1+\infty}=\mathrm{Span}(z^mD^a\;\lvert\; 0\leq a\leq (i+2)/2,m\in\BZ_{\geq 1}).
\]
We denote by $\Gr^F_{\bullet} \!W_{1+\infty}^+$ the associated graded Lie algebra.  From \eqref{expl_brack}, the induced Lie bracket on $\Gr^F_{\bullet} \!W_{1+\infty}^+$ is nontrivial; it is given  by
\begin{equation}
\label{GrW}
[z^mD^a,z^nD^b]=(an-bm)z^{m+n}D^{a+b-1}.
\end{equation}
So the associated graded Lie algebra $\Gr^F_{\bullet} \!W_{1+\infty}^+$ is isomorphic to (part of) the \textit{quantum torus} associated to the lattice $\BZ^2$ with skew-symmetric form $\omega((m,a),(n,b))=an-bm$.
\begin{thmx}[Theorem \ref{fder} and Corollary \ref{RH}]
\label{ThmA}
There is an isomorphism of Lie algebras
\[
\hat{\fg}_{Q^{(3)},\tilde{W}}\cong \Gr^F_{\bullet}\!\Walg_{1+\infty}^+
\]
between the affine BPS Lie algebra\footnote{As defined in \S \ref{CopSec}.} for the pair $(Q^{(3)},\tilde{W})$ and the associated graded Lie algebra of $\Walg_{1+\infty}^+$ with respect to the filtration $F_{\bullet}$.  As a result, there is an isomorphism of algebras
\[
\CoHA_{\BA^2}\cong \UEA(\Gr^F_{\bullet}\!\Walg_{1+\infty}^+).
\]
\end{thmx}
The algebras $\Coha^{T}_{\BA^2}$ (for various $T$) have appeared in a great deal of work in geometric representation theory in the last decade.  One often restricts to the \textit{spherical subalgebra} of $\Coha^{T}_{\BA^2}$, i.e. the algebra generated by the subspace $\Coha^{T}_{\BA^2,1}$, corresponding to coherent sheaves on $\BA^2$ of length one.  One of the main consequences of Theorem \ref{ThmA} is the following:
\begin{thmx}[Theorem \ref{TSG}]
Let $T$ act on the quiver $Q^{(3)}$, leaving the potential $\tilde{W}$ invariant.  Then the CoHA $\CoHA_{\BA^2}^T$ is spherically generated, i.e. it is generated by the piece $\CoHA_{\BA^2,1}^T$ corresponding to coherent sheaves of length one.  In particular, $\CoHA_{\BA^2,1}$ is spherically generated.
\end{thmx}
It may come as a temporary disappointment to the reader that the associated graded algebra of the filtered algebra $W^+_{1+\infty}$ appears in Theorem \ref{ThmA}, as opposed to $W^+_{1+\infty}$ itself.  This is, however, somewhat inevitable: the algebra $\CoHA_{Q^{(3)},\tilde{W}}$ is \textit{bi}graded, with one grading coming from the decomposition of the stack of finite-dimensional $\mathbb{C}[x,y]$-modules according to their dimension, and one grading coming from cohomological degree.  The Lie algebra $W^+_{1+\infty}$, by contrast, is graded according to the exponent of $z$, but only \textit{filtered} according to the exponent of $D$.
\smallbreak
To improve upon this situation, we let $\mathbb{C}^*$ act on $\mathbb{A}^2$ by scaling the first coordinate, and leaving the second invariant.  This induces a $\mathbb{C}^*$-action on $\Mst_d(\BC[x,y])$ for all $d$, and we take the equivariant Borel--Moore homology of this stack with respect to this action, and define the cohomological Hall algebra
\[
\CoHA_{\BA^2}^{\BC^*}\coloneqq \bigoplus_{d\in\BN}\HO^{\BM}(\Mst^{\BC^*}_d(\BC[x,y]),\BQ)
\]
to be this equivariant Borel--Moore homology, equipped with the usual multiplication via correspondences.  The algebra $\BQ[t]\cong \HO_{\BC^*}\coloneqq \HO(\B \BC^*,\BQ)$ acts on $\CoHA_{\BA^2}^{\BC^*}$, with the action induced by the $\BC^*$-action on $\BA^2$, and the Hall algebra multiplication is $\HO_{\BC^*}$-linear.  
\smallbreak
We denote by
\begin{equation}
\label{Rees}
\Rees_F[W_{1+\infty}^+]\coloneqq\sum_{i\geq -1}F_{2i}W_{1+\infty}^+t^{i}\BQ[t]\subset W_{1+\infty}^+\otimes_{\BQ} t^{-1}\BQ[t]
\end{equation}
the Rees Lie algebra.  This is a $\BQ[t]$-linear Lie algebra, and so may be thought of as a family of Lie algebras over the affine line, deforming the algebra $\Gr^F_{\bullet}\!W_{1+\infty}^+$ to $W_{1+\infty}^+$: for $\epsilon\neq 0$ the specialisation at $t=\epsilon$ is isomorphic to $W^+_{1+\infty}$, and the specialisation at $t=0$ is isomorphic to $\Gr^F_{\bullet}W^+_{1+\infty}$.
\smallbreak
Let $A$ be a commutative ring.  For $\fg$ an $A$-linear Lie algebra, we define the tensor algebra
\[
T_{A}(\fg)\coloneqq \bigoplus_{m\geq 0} \underbrace{\fg\otimes_{A}\ldots\otimes_{A}\fg}_{m \;\;\mathrm{ times}}
\]
and define 
\[
\UEA_{A}(\fg)=T_{A}(\fg)/\langle \alpha\cdotsh \beta-\beta\cdotsh\alpha-[\alpha,\beta]_{\fg}\;\lvert\; \alpha,\beta\in\fg\rangle.
\]

\begin{thmx}[Corollary \ref{HW_cor}]
Let $\BC^*$ act on $Q^{(3)}$ with weights $1,0,-1$ on the arrows $a,a^*,\omega$ respectively.  There is an isomorphism of $\HO_{\BC^*}=\BQ[t]$-linear Lie algebras
\[
\hat{\fg}_{Q^{(3)},\tilde{W}}^{\BC^*}\cong \Rees_F[W^+_{1+\infty}]
\]
between the deformed affine BPS Lie algebra for $(Q^{(3)},\tilde{W})$, introduced in \S \ref{CopSec}, and the Rees Lie algebra of $W^+_{1+\infty}$.  As a result, there is an isomorphism of $\BQ[t]$-linear algebras:
\[
\Coha^{\BC^*}_{\BA^2}\cong \UEA_{\BQ[t]}(\Rees_F[W^+_{1+\infty}]).
\]
\end{thmx}
\medbreak
Let $T=(\BC^*)^2$ act on $\BA^2$, by scaling the two coordinates independently.  From the above results, it is very natural to seek an algebraic description of the fully deformed CoHA $\CoHA^{T}_{\BA^2}$.  We denote by $\CY_{t_1,t_2,t_3}(\widehat{\mathfrak{gl}(1)})^+$ the strictly positive part of the affine Yangian; a presentation is recalled in \S \ref{SHY_sec}.  Our final theorem is a strengthening of a result of Rap{\v{c}}{\'a}k, Soibelman, Yang and Zhao \cite{RSYZ20}:% (in fact it is an immediate consequence of their result, along with spherical generation):
\begin{thmx}[Theorem \ref{Full_Yang}]
For $T=(\mathbb{C}^*)^2$ as above, there is an isomorphism of $\HO_T$-linear algebras
\[
\CoHA^{T}_{\BA^2}\cong \CY_{t_1,t_2,t_3}(\widehat{\mathfrak{gl}(1)})^+.
\]
\end{thmx}

%\smallbreak
%For $Q$ a finite type ADE quiver, the same methods can be used to prove the following.
%\begin{thmx}
%\label{ThmB}
%Let $Q$ be a finite type ADE quiver, and let $\tilde{Q}$ be the corresponding tripled quiver, with canonical cubic potential
%\[
%\tilde{W}=(\sum_{a\in Q_1}[a,a^*])(\sum_{i\in Q_0}\omega_i).
%\]
%Then there is an isomorphism of algebras
%\[
%\CoHA_{Q^{(3)},\tilde{W}}\cong \UEA(\fn^-_Q\otimes \BQ[z])
%\]
%where $\fn^-$ is the negative half of the simple Lie algebra associated to the underlying Dynkin diagram of $Q$.
%\end{thmx}

\subsection{Notation and conventions}
\begin{itemize}
\item
Every cohomological Hall algebra that we consider throughout the paper will carry a $\HO(\B\BC^*,\BQ)$-action arising from the morphism from the stack of finite-dimensional modules over an algebra to $\B \BC^*$ given by taking a module $\rho$ to $\wedge^{\mathrm{top}}\rho$.  We denote this copy of $\B\BC^*$ by $\B \BC^*_u$, to distinguish it from other instances of $\B \BC^*$ arising from e.g. torus actions on quivers.  We continue to write, as in the introduction, $\HO_{\BC_u^*}\coloneqq\HO(\B \BC^*_u,\BQ)\cong\BQ[u]$, with $u$ in cohomological degree $2$.  
\item
More generally, if $T=(\BC^*)^l$ is a torus, we define $\HO_T\coloneqq \HO(\B T,\BQ)$.
\item
If a torus $T$ acts on a stack $\FX$ with quotient $\FY$, we define $({}^{\fp'}\!\tau^{\leq i}, {}^{\fp'}\!\tau^{\geq i})$ to be the shift of the usual perverse t structure on the derived category of constructible complexes on $\FY$, so that $\CF$ is in the heart if and only if $(\FX\rightarrow \FY)^*\CF$ is perverse (see \S \ref{toolkit}).
\item
We define $\BN\coloneqq\BZ_{\geq 0}$.
\end{itemize}

\subsection{Acknowledgements}
Special thanks go to Ezra Getzler, who encouraged me in this endeavour along with many others over many years.  Thanks go to Kevin Costello, Brian Williams, Sasha Minets and Shivang Jindal for helpful conversations about the results of this paper, and encouragement to write them down correctly.  Thanks also to Andrei Negu\c{t} for numerous helpful comments regarding an earlier draft, and Lucien Hennecart for spotting a typo in an advanced draft.  Thanks are also due to KIAS, and Minhyong Kim, for organizing an extremely stimulating stay in Seoul in Summer 2022, where this paper was finally written.  Finally, many thanks are due to the anonymous referee for their many useful corrections and suggestions.

\section{Algebraic background}
\subsection{Heisenberg algebras}
We denote by $\Heis$ the Lie algebra over $\BQ$ having basis $\{p,q,c\}$, subject to the relations
\begin{align*}
[q,p]=&c\\
[c,p]=[c,q]=&0.
\end{align*}
We denote by $\Heis\lmod$ the category of modules over $\Heis$.  We give $\Heis$ a (cohomological) $\BZ$-grading, putting $p,c,q$ in degrees $2,0,-2$ respectively.  A cohomologically graded $\Heis$-module is a $\Heis$-module $\rho$, along with a direct sum decomposition according to cohomological degree of the underlying vector space
\[
\rho=\bigoplus_{n\in\BZ}\rho^n
\]
such that $p$ maps $\rho^n$ to $\rho^{n+2}$, $q$ maps $\rho^n$ to $\rho^{n-2}$, and $c$ preserves the decomposition.  We denote by $\Heis\lmod_{\BZ}$ the category of cohomologically graded $\Heis$-modules.  A morphism of cohomologically graded $\Heis$-modules is a morphism of the underlying $\Heis$-modules that preserves the grading.  If $\rho$ is a cohomologically graded $\Heis$ module, and $n\in\BZ$, we denote by $\rho[n]$ the cohomologically graded $\Heis$-module obtained by setting $\rho[n]^i=\rho^{n+i}$.  
\smallbreak
For a quiver $Q$ with vertex set $Q_0$, we define the Lie algebra
\[
\Heis_Q'\coloneqq \bigoplus_{i\in Q_0}\Heis.
\]
For $i\in Q_0$ we denote by $p_i,q_i,c_i$ the basis elements in $\Heis'_Q$ corresponding to the $i$th summand.  We define the Lie subalgebra $\Heis_Q\subset \Heis'_Q$ to be the span of $\{p_i\;\lvert\; i\in Q_0\}\cup \{c_i\;\lvert\; i\in Q_0\}\cup \{q_{\Sigma}\coloneqq \sum_{i\in Q_0}q_i\}$.

For $\cc\in \mathbb{Z}$ we say that a $\Heis$-representation $\rho$ has central charge $\cc$ if $c$ acts on $\rho$ via scalar multiplication by $\cc$.  For $\cc\in \mathbb{Z}^{Q_0}$ we say that a $\Heis_Q$-representation $\rho$ has central charge $\cc$ if $c_i$ acts by multiplication by $\cc_i$.  There is a canonical inclusion 
\begin{equation}
\label{Heis_forg}
\Heis\hookrightarrow \Heis_Q
\end{equation}
sending $p\mapsto \sum_{i\in Q_0}p_i$, $q\mapsto q_{\Sigma}$ and $c\mapsto \sum_{i\in Q_0}c_i$.  We can consider a $\Heis_Q$-module as a $\Heis$-module via this inclusion.  Placing all $p_i,c_i,q_i$ in degrees $2,0,-2$ respectively, we give $\Heis_Q$ a $\BZ$-grading, and define the category $\Heis_Q\lmod_{\BZ}$ of graded modules as before.
\smallbreak
The vector space $\BQ_0=\BQ$ forms a $\Heis_Q$-module of central charge $0$ if we let all $p_i,q_i,c_i$ act via the zero matrix.  The categories $\Heis_Q\lmod$ and $\Heis_Q\lmod_{\BZ}$ are symmetric tensor categories with monoidal unit $\BQ_0$.  If $\mathscr{C}$ is a tensor category, a $\mathscr{C}$\textit{-algebra object} is an object $\CF$ in $\mathscr{C}$, along with morphisms $\mathbf{1}_{\mathscr{C}}\rightarrow \CF$ from the monoidal unit, and a multiplication morphism $\CF\otimes \CF\rightarrow \CF$, satisfying the standard properties.  If $\mathscr{C}$ is a symmetric tensor category, a $\mathscr{C}$\textit{-Lie algebra object} is an object $\CF$ in $\mathscr{C}$ along with an antisymmetric morphism $\CF\otimes\CF\rightarrow \CF$ satisfying the Jacobi identity.  If $\rho$ is a Lie algebra object in $\Heis_Q$, the universal enveloping algebra $\UEA(\rho)$ is an algebra object in $\Heis_Q$, with underlying algebra the universal enveloping algebra of the underlying Lie algebra of $\rho$.
\smallbreak
We say that a $\Heis_Q$-module $\rho$ is integrable if there is a decomposition
\begin{equation}
\label{ccdecomp}
\rho=\bigoplus_{\dd\in\BN^{Q_0}}\rho_{[\dd]}
\end{equation}
with each $\rho_{[\dd]}$ a $\Heis_Q$-module of central charge $\dd$.  Similarly, we say that a $\Heis$-module $\rho$ is integrable if it admits a decomposition $\rho=\bigoplus_{d\in\BN}\rho_{[d]}$ with $\rho_{[d]}$ of central charge $d$.
\smallbreak
Let $\Heis_Q$ act on the $\BN^{Q_0}$-graded algebra $\mathcal{A}$, where the $\BN^{Q_0}$-grading is induced by a decomposition of $\CA$ into sub-modules of fixed central charge, as in \eqref{ccdecomp}.  Then $\CA$ is an algebra object in $\Heis_Q$ if each $p_i$ for $i\in Q_0$, and $q_{\Sigma}$, act by derivations on $\CA$.  Similarly, if $\fn$ is a $\BN^{Q_0}$-graded Lie algebra, and also an integrable $\Heis_Q$-module for which the $\BN^{Q_0}$-graded decomposition is provided by the decomposition according to central charge, then $\fn$ is a Lie algebra object in $\Heis_Q$ if and only if for every $i\in Q_0$, $p_i$ and $q_{\Sigma}$ act by Lie algebra derivations.
\smallbreak
%{\textcolor{red}{ May be deleted...}}
%Let $\fn$ be a $\BN^{Q_0}$-graded Lie algebra, and let $\fn\otimes\BQ[z]$ denote the $\BQ[z]$-linear extension, i.e. the Lie algebra with bracket $[\alpha\otimes z^m,\beta\otimes z^n]=[\alpha,\beta]z^{m+n}$.  We define a $\Heis_Q$-action as follows: if $\alpha\in\fn_{\dd}$ then
%\begin{align*}
%p_i\cdot(\alpha\otimes z^m)&=\dd_i (\alpha\otimes z^{m+1})\\
%q_i\cdot(\alpha\otimes z^m)&=m(\alpha\otimes z^{m-1})\\
%c_i\cdot(\alpha\otimes z^m)&=\cc_i(\alpha\otimes z^m).
%\end{align*}
%The $\dd$-weight space of $\fn$ for $\fn$ is then the $\dd$-weight space for the given grading on $\fn$.
%\smallbreak

\subsection{The Lie algebra $W_{1+\infty}^+$}
Recall from \eqref{Wplus} the definition of $W^+_{1+\infty}\subset W_{1+\infty}$.  The Lie bracket on $W^+_{1+\infty}$ is given by
\begin{equation}
\label{expl_brac}
[z^mf(D),z^ng(D)]=z^{m+n}(f(D+n)g(D)-f(D)g(D+m)).
\end{equation}
We set 
\[
\tilde{D}^{(a)}\coloneqq D^at^{a-1}\in \Rees_F[W^+_{1+\infty}].
\]
The Lie algebra $\Rees_F[W^+_{1+\infty}]$ (defined as in \eqref{Rees}) is spanned by elements $z^m\tilde{D}^{(a)}t^r$ with $m\geq 1, a\geq 0, r\geq 0$.  We give $z^m\tilde{D}^{(a)}t^r=z^mD^at^{r+a-1}$ cohomological degree $2r+2a-2$, making $\Rees_F[W^+_{1+\infty}]$ into a cohomologically graded Lie algebra.  

%Recall that we give $\Heis$ a cohomological grading by putting $p,c,q$ in cohomological degrees $2,0,-2$ respectively.  
\begin{proposition}
\label{Dinj}
\begin{itemize}
\item
There is an action of $\Heis$ on $W^+_{1+\infty}$ by derivations, given by sending $p$ to $[D^2,-]$ and sending $q$ to the endomorphism
\begin{align*}
\partial_D\colon &W^+_{1+\infty}\rightarrow W^+_{1+\infty}\\
&z^mf(D)\mapsto z^mf'(D).
\end{align*}
For $m\geq 1$, the subspace $\mathrm{Span}(z^mD^n\;\lvert \; n\in \BN)$ is preserved by this action, on which it has central charge $2m$.
\item
There is an action of $\Heis$ on $\Rees_F[W^+_{1+\infty}]$ by derivations, sending $p$ to $t[D^2,-]$ and $q$ to the endomorphism $t^{-1}\partial_D$, making $\Rees_F[W^+_{1+\infty}]$ into a cohomologically graded $\Heis$-module.  For $m\geq 1$, the subspace \\$\mathrm{Span}(z^m\tilde{D}^nt^r\;\lvert \; n,r\in \BN)$ is preserved by this action, on which it has central charge $2m$.
\end{itemize}
\end{proposition}
\begin{proof}
We prove the first part, the second is similar.  The fact that $[D^2,-]$ is a derivation follows from the Jacobi identity, while the fact that $\partial_D$ is a derivation follows from the formula \eqref{expl_brac}.  Now we calculate
\begin{align*}
[\partial_D,[D^2,-]](z^mD^n)&=\partial_D[D^2,z^mD^n]-[D^2,nz^mD^{n-1}]\\
&=[2D,z^mD^n]+[D^2,nz^mD^{n-1}]-[D^2,nz^mD^{n-1}]\\
&=2z^m((D+m)D^n-D^{1+n})\\
&=2mz^mD^n
\end{align*}
as required.
\end{proof}
For $m\in2 \cdotsh\BZ_{\geq 1}$ we denote by $V_{[m]}\subset W^+_{1+\infty}$ the $\Heis$-submodule with central charge $m$.  By the above proposition, it has a basis given by elements $z^{m/2}D^n$.  If we let $\BQ[u]$ act on $V_{[m]}$ by letting $u$ act via the raising operator $p=[D^2,-]$, then all of the $V_{[m]}$ are isomorphic to $\BQ[u]$, considered as $\BQ[u]$-modules.  
\smallbreak
We endow $V_{{[2m]}}$ with a filtration by setting 
\begin{align*}
F_iV_{{[2m]}}=\begin{cases} \mathrm{Span}_{\BQ}(z^{m},z^mD,\ldots, z^mD^{i/2+1})&\textrm{if }i\textrm{ is even}\\
F_{i-1}V_{{[2m]}}&\textrm{if }i\textrm{ is odd.}\end{cases}
\end{align*}
So if $i$ is odd, the $i$th piece of the associated graded object $\Gr^F_i\!V_{{[2m]}}$ is zero, and if $i\in \BZ_{\geq -1}$ then a basis of $\Gr^F_{2i}\!V_{{[2m]}}$ is given by $\{z^{m}D^{i+1}\}$.  We have $pF_iV_{[2m]}\subset F_{i+2}V_{[2m]}$ and $qF_iV_{[2m]}\subset F_{i-2}V_{[2m]}$, so that $\Gr^F_{\bullet}\!V_{[2m]}$ becomes a cohomologically graded $\Heis$-module, where the cohomological degree of $z^mD^n$ is $2n-2$.  Taking the direct sum of the filtrations above, we obtain the \textit{order} filtration $F_{\bullet}$ on $W^+_{1+\infty}$, considered in the introduction.  The Lie bracket on $W^+_{1+\infty}$ preserves the order filtration, and there is an isomorphism of Lie algebra objects in the category $\Heis\lmod_{\BZ}$:
\begin{align*}
\Psi\colon \Gr^F_{\bullet}\!W^+_{1+\infty}&\rightarrow \Rees_F[W^+_{1+\infty}]\otimes_{\BQ[t]}(\BQ[t]/(t))\\
z^mD^n&\mapsto z^m\tilde{D}^{(n)}\otimes 1.
\end{align*}
\smallbreak

\begin{lemma}
\label{SphGen}
The Lie algebras $\Gr^F_{\bullet}\!W^+_{1+\infty}$ and $\Rees_F[W^+_{1+\infty}]$ are spherically generated: they are generated by their respective $\Heis$-submodules of central charge $2$.
\end{lemma}
\begin{proof}
Let $\fg\subset \Rees_F[W^+_{1+\infty}]$ be the Lie subalgebra generated by $\Rees_F[W^+_{1+\infty}]_{[2]}=\mathrm{Span}(z\tilde{D}^{(a)}t^r\;\lvert \; a\geq 0,r\geq 0)$.  For $a,m\geq 1$, the identity \eqref{expl_brack} yields
\begin{align}
\nonumber [z\tilde{D}^{(a)},z^m\tilde{D}^{(0)}]=&t^{a-2}z^{m+1}((D+m)^a-D^a)\\
=&z^{m+1}(ma\tilde{D}^{(a-1)}+\textrm{lower order terms in }\tilde{D}).\label{prol}
\end{align}
We induct on $m$, and assume that $z^{m'}\tilde{D}^{(a')}\in \Rees_F[W^+_{1+\infty}]$ for every $1\leq m'<m$ and $a'\geq 0$.  Substituting $m\mapsto m-1$ and $a\mapsto 1$ in \eqref{prol} we deduce that $z^m\tilde{D}^{(0)}\in\mathfrak{g}$.  Then we induct on $a$, and assume that $z^{m}\tilde{D}^{(a')}\in \Rees_F[W^+_{1+\infty}]$ for every $0\leq a'<a$.  Substituting $a\mapsto a+1$ and $m\mapsto m-1$ in \eqref{prol} we deduce that $z^m\tilde{D}^{(a)}\in\mathfrak{g}$ as required for the inductive step.

Spherical generation of $\Gr^F_{\bullet}\!W^+_{1+\infty}$ follows directly from \eqref{GrW}, or from the spherical generation of $\Rees_F[W^+_{1+\infty}]$ and the isomorphism $\Psi$.
\end{proof}

We record a variant of the result, for later.  It follows from \eqref{GrW}.
\begin{lemma}
\label{Blawan}
The elements, defined for all $n,m\geq 0$ by
\[
\begin{cases}
(\mathbf{ad}_{zD})^{m}(zD^n)& \textrm{if }m<n \textrm{ or } n=0\\
((\mathbf{ad}_{zD})^{m-n}\circ \mathbf{ad}_{D}\circ (\mathbf{ad}_{zD})^{n-1})(zD^{n+1}) & \textrm{otherwise}
\end{cases}
\]
provide a basis for $W^+_{1+\infty}$.
\end{lemma}
The Lie algebra $\Gr^F_{\bullet} \!W^+_{1+\infty}$ carries two gradings, one by central charge, and one given by the cohomological grading, and $\Gr^F_{\bullet} \!W^+_{1+\infty}$ is a Lie algebra object in $\Heis\lmod_{\BZ}$.  It turns out that this object satisfies a simple universal property: the Lie algebra $\Gr^F_{\bullet} \!W^+_{1+\infty}$ is the universal Lie algebra object in $\Heis\lmod_{\BZ}$ containing $\Gr^F_{\bullet}\!V_{[2]}$, and containing only representations $\Gr^F_{\bullet}\!V_{[m]}$ for various $m$ as subrepresentations (i.e. no copies of $\Gr^F_{\bullet}V_{[m]}[j]$ with $j\neq 0$)).  More precisely,
\begin{lemma}
\label{univ_prop}
If $\mathcal{H}$ is a Lie algebra object in $\Heis\lmod_{\BZ}$ containing only unshifted representations $\Gr^F_{\bullet}\! V_{[m]}$ as cohomologically graded $\Heis$-subrepresentations, and $\iota\colon \Gr^F_{\bullet} \! V_{[2]}\rightarrow \mathcal{H}$ is an inclusion in $\Heis\lmod_{\BZ}$, then $\iota$ extends uniquely to a morphism of $\Heis$-Lie algebra objects $\Gr^F_{\bullet}\! W^+_{1+\infty}\rightarrow \mathcal{H}$.
\end{lemma}
\begin{proof}
The lemma is equivalent to the claim that $\Gr^F_{\bullet}\!W^+_{1+\infty}$ is isomorphic to $\mathcal{L}'=\mathcal{L}/\mathcal{I}$, where $\mathcal{L}$ is the free $\mathbb{Z}$-graded Lie algebra $\mathcal{L}$ generated by $\Gr^F_{\bullet}\!V_{[2]}$, and $\mathcal{I}$ is the Lie ideal containing all the $\BZ$-graded $\Heis$-submodules of $\mathcal{L}$ that are not isomorphic to unshifted copies of $\Gr^F_{\bullet}\!V_{[m]}$ for various $m$.  
\smallbreak
The Lie algebra $\mathcal{L}'$ carries a grading by tensor degree, and we claim that $\mathcal{L}'_m\cong \Gr^F_{\bullet}\!V_{[2m]}$, where on the left hand side of this isomorphism the subscript denotes the degree with respect to the tensor degree grading.  This obviously implies the lemma.  We prove the claim by induction.  The base case $m=1$ is trivial, and so we assume the claim is proved for $i< m$, with $m\geq 2$.  By the inductive hypothesis, the $m$th summand $\mathcal{L}'_m$ is either a quotient of 
\[
\Gr^F_{\bullet}\!V_{[2]}\wedge \Gr^F_{\bullet}\!V_{[2]}\cong \bigoplus_{i\geq 0}\Gr^F_{\bullet}\!V_{[4]}[-4i]
\]
if $m=2$, or 
\[
\Gr^F_{\bullet}\!V_{[2]}\otimes \Gr^F_{\bullet}\!V_{[2m-2]}\cong \bigoplus_{i\geq 0}\Gr^F_{\bullet}\!V_{[2m]}[2-2i]
\]
if $m\geq 3$.  In either case, at most a single unshifted copy of $\Gr^F_{\bullet}\! V_{[2m]}$ survives after we quotient out by all of the shifted copies of $\Gr^F_{\bullet}\!V_{[2m]}$.  On the other hand, $\CL'_m$ contains at least one copy of  $\Gr^F_{\bullet}\!V_{[2m]}$, since by its universal property, and spherical generation of $\Gr^F_{\bullet}\!W^+_{1+\infty}$, there is a surjective morphism $\CL'\rightarrow \Gr^F_{\bullet}\!W^+_{1+\infty}$.  The inductive claim follows.
\end{proof}
Given an integrable cohomologically-graded $\Heis$-module $\rho$, we define the characteristic function
\[
\chi(\rho)\coloneqq \sum_{i,j\in\mathbb{Z}}\dim(\rho_{[i]}^j)v^iq^{j},
\]
where $\rho_{[i]}^j$ denotes the $j$th cohomologically graded piece of the summand of $\rho$ with central charge $i$.
\begin{corollary}
\label{gen_cor}
Let $\mathcal{H}$ be an integrable $\BZ$-graded $\Heis$-module generated as a Lie algebra by $\Gr^F_{\bullet}\!V_{[2]}\subset \mathcal{H}$ and satisfying
\begin{align*}
\chi(\mathcal{H})=&\chi(\Gr^F_{\bullet}\!W^+_{1+\infty})\\
=&\sum_{i\geq 1,j\geq -1}v^{2i}q^{2j}\\
=&v^2q^{-2}(1-v^2)^{-1}(1-q^2)^{-1}.
\end{align*}
Then $\mathcal{H}\cong\Gr^F_{\bullet}\!W^+_{1+\infty}$. 
\end{corollary}
\begin{proof}
By the universal property of $\Gr^F_{\bullet}\!W^+_{1+\infty}$ (Lemma \ref{univ_prop}) there is a morphism of Lie algebra objects in $\Heis\lmod_{\BZ}$
\[
f\colon \Gr^F_{\bullet}\!W^+_{1+\infty}\rightarrow \CH
\]
sending the copy of $\Gr^F_{\bullet}\!V_{[2]}$ inside $\Gr^F_{\bullet}\!W^+_{1+\infty}$ isomorphically to the given copy inside $\CH$.  By supposition, $f$ is surjective.  On the other hand, the source and the target have the same characteristic function, so $f$ is an isomorphism.
\end{proof}

\section{DT theory background}

\subsection{Vanishing cycle toolkit}
\label{toolkit}
If $\FX$ is a smooth stack, we denote by $\BQ_{\FX}^{\vir}=\BQ_{\FX}[\dim(\FX)]$ the constant perverse sheaf on $\FX$.  More generally, if an extra torus $T$ acts on a smooth stack $\FX$ with $T$-quotient $\FY$, we define\footnote{The motivation for this slightly mixed convention is that we would like $\BQ_{\FY}^{\vir}$ to correspond to a $T$-equivariant perverse sheaf on $\FY$, i.e. we would like the pullback $(\FX\rightarrow \FY)^*\BQ_{\FY}^{\vir}$ to be perverse.} $\BQ_{\FY}^{\vir}=\BQ_{\FY}[\dim(\FX)]$.  Under the same conditions, we modify the usual perverse truncation functors for constructible complexes of sheaves on $\FY$ by setting ${}^{\fp'}\!\!\tau^{\leq i}={}^{\fp}\tau^{\leq i-\dim(T)}$.
\smallbreak
Given a function $f$ on a smooth stack $X$, we define the vanishing cycle functor $\phip{f}$ as in \cite{KSsheaves}.  It is an exact functor with respect to the perverse t structure, and in particular $\phip{f}\BQ_{X}^{\vir}$ is a perverse sheaf on $X$, which we may refer to just as the sheaf of vanishing cycles on $X$.  We will often abuse notation by just writing $\HO(X,\phip{f})$ instead of $\HO(X,\phip{f}\BQ_{X}^{\vir})$.
\smallbreak
Let $f,g$ be functions on smooth stacks $X,Y$ respectively.  The Thom--Sebastiani theorem \cite{Ma01} states that there is a natural bifunctorial equivalence of constructible complexes of sheaves on $X\times Y$
\[
\phip{f\boxplus g}(\CF\boxtimes\CG)\lvert_{f^{-1}(0)\times g^{-1}(0)}\cong\phip{f}\CF\boxtimes\phip{g}\CG.
\]
Assuming that $\supp(\phip{f\boxplus g}(\CF\boxtimes\CG))\subset f^{-1}(0)\times g^{-1}(0)$, applying the derived global sections functor to the above isomorphism yields the isomorphism
\[
\TS\colon \HO(X,\phip{f}\CF)\otimes \HO(Y,\phip{g}\CG)\cong \HO(X\times Y,\phip{f\boxplus g}(\CF\boxtimes \CG)),
\]
the analogue in vanishing cycle cohomology of the K\"unneth isomorphism in ordinary cohomology.
\smallbreak
Given a morphism $q\colon X\rightarrow Y$ and a function $f$ on $Y$ there is a natural transformation $\theta\colon\phip{f}q_*\rightarrow q_*\phip{fq}$, which is an isomorphism if $q$ is projective.  The natural transformation $\theta$ induces the natural transformation $\eta\colon\phip{f}\rightarrow q_*\phip{f}q^*$ by adjunction.  Applying the derived global sections functor, $\eta$ induces the pullback morphism
\begin{equation}
\label{pbmorph}
q^{\star}\colon \HO(Y,\phip{f}\CF)\rightarrow \HO(X,\phip{fq}q^*\CF).
\end{equation}
Given a proper morphism $p\colon X\rightarrow Y$ between smooth stacks and a function $f$ on $Y$, we obtain the pushforward morphism
\begin{equation}
\label{pfmorph}
p_{\star}\colon \HO(X,\phip{pf}\BQ_{X})[2\dim(X)]\rightarrow \HO(Y,\phip{f}\BQ)[2\dim(Y)]
\end{equation}
by applying $\phip{f}$ to $p_*\BQ_{X}[2\dim(X)]\rightarrow \BQ_{Y}[2\dim(Y)]$ (the Verdier dual of $\BQ_Y\rightarrow p_*\BQ_X$), and composing with $\theta^{-1}$.
\smallbreak
Let $\FX$ be a smooth stack, $f$ be a function on $\FX$, and $\CF\in\Dbc(\FX)$ be a constructible complex.  Let $\Delta\colon \FX\rightarrow \FX\times \FX$ be the diagonal embedding.  Then $f=(f\boxplus 0)\circ\Delta$ and there is a natural isomorphism $\CF\cong \Delta^*(\CF\boxtimes \BQ_{\FX})$.  The pullback morphism $\Delta^{\star}$, composed with the Thom-Sebastiani isomorphism, yields
\[
\HO(\FX,\phip{f}\CF)\otimes\HO(\FX,\BQ)\rightarrow \HO(\FX,\phip{f}\CF),
\]
and an action of the cohomology of $\FX$ on $\HO(\FX,\phip{f}\CF)$.  If $f=0$ this is just the usual action of $\HO(\FX,\BQ)$ on $\HO(\FX,\CF)$, and if in addition $\CF=\BQ_{\FX}$ this is the usual associative algebra structure on $\HO(\FX,\BQ)$.  If $q\colon \FX\rightarrow \FY$ is a morphism of smooth stacks, from the commutativity of
\[
\xymatrix{
\FX\ar[r]^q\ar[d]_{(\id\times q)\circ\Delta}&\FY\ar[d]^{\Delta}\\ \FX\times \FY\ar[r]^{q\times \id}&\FY\times \FY
}
\]
we deduce that \eqref{pbmorph} respects the $\HO(\FY,\BQ)$-actions on the domain and target.  If $q$ is projective, one may show as in \cite[Prop.2.14]{Da13} that \eqref{pfmorph} respects the $\HO(\FY,\BQ)$ actions on the domain and target.
\subsection{Quivers}
\label{Quiver_sec}
Throughout the paper, $Q$ will denote a finite quiver, i.e. a pair of finite sets $Q_1$ and $Q_0$ (the arrows and vertices, respectively) along with a pair of morphisms $s,t\colon Q_1\rightarrow Q_0$.  The free path algebra $kQ$ of $Q$ over the field $k$ is the algebra having as $k$-basis the paths in $Q$, including ``lazy'' paths $e_i$ of length zero at each of the vertices $i\in Q_0$.
\smallbreak
We define the Euler form
\begin{align*}
\chi_Q\colon &\BN^{Q_0}\times\BN^{Q_0}\rightarrow \BZ\\
&(\dd,\ee)\mapsto(\dd,\ee)_Q\coloneqq\sum_{i\in Q_0}\dd_i\ee_i-\sum_{a\in Q_1}\dd_{s(a)}\ee_{t(a)}.
\end{align*}
\smallbreak
The \textit{double} of $Q$ is obtained by setting $\overline{Q}_0=Q_0$ and $\overline{Q}_1=Q_1\coprod Q_1^*$, where $Q_1^*\coloneqq \{ a^* \;\lvert\; a\in Q_1\}$ and the orientation of $a^*$ is the opposite of the orientation of $a$.  The preprojective algebra is defined by
\[
\Pi_Q\coloneqq \CC \overline{Q}/\langle \sum_{a\in Q_1} [a,a^*]\rangle.
\]
The quiver $\tilde{Q}$ is defined by adding a loop $\omega_i$ to each vertex $i\in \overline{Q}_0$.
\begin{example}
If $Q=Q^{(1)}$ is the Jordan quiver with one vertex and one loop $a$, then $\tilde{Q}=Q^{(3)}$ is the three loop quiver, with one vertex and loops labelled $a,a^*,\omega$.
\end{example}

A \textit{potential} $W\in\CC Q/[\CC Q,\CC Q]_{\mathrm{vect}}$ is a linear combination of cyclic words in $Q$.  If $W=a_1\ldots a_l$ is a single cyclic word then
\[
\partial W/\partial a\coloneqq\sum_{a_i=a}a_{i+1}\ldots a_la_1\ldots a_{i-1}.
\]
We extend this definition to general potentials by linearity.  The \textit{Jacobi algebra} associated to the pair $(Q,W)$ is defined by
\[
\Jac(Q,W)\coloneqq \CC Q/\langle \partial W/\partial a \;\lvert \; a\in Q_1\rangle.
\]
If $Q$ is an arbitrary quiver, the tripled quiver $\tilde{Q}$ carries the \textit{canonical cubic potential} 
\[
\tilde{W}\coloneqq\left(\sum_{i\in Q_0}\omega_i\right)\left(\sum_{a\in Q_1}[a,a^*]\right).
\]
Let $\Pi_Q[\omega]$ denote the polynomial algebra in one variable $\omega$ with coefficients in $\Pi_Q$.  There is an isomorphism
\begin{align}
\label{TJA}
j\colon &\Pi_Q[\omega]\rightarrow \Jac(\tilde{Q},\tilde{W})
\end{align}
defined by $j(a)=a$, $j(a^*)=a^*$, $j(\omega)=\sum_{i\in Q_0} \omega_i$.

\subsection{Moduli spaces}
Given a quiver $Q$ and dimension vector $\dd\in\BN^{Q_0}$ we define
\[
\AA_{Q,\dd}\coloneqq\prod_{a\in Q_1}\Hom(\CC^{\dd(s(a))},\CC^{\dd(t(a))});\quad \quad \Gl_{\dd}\coloneqq \prod_{i\in Q_0} \Gl_{\dd(i)}.
\]
The group $\Gl_{\dd}$ acts on $\AA_{Q,\dd}$ by conjugation.  The moduli stack $\Mst_{\dd}(Q)$, when evaluated on a scheme $X$, yields the groupoid of coherent sheaves on $X$, flat over $X$, equipped with an action of $\BC Q$, such that at every geometric $K$-point of $X$, with $K\supset \mathbb{C}$, the induced $KQ$-module is $\dd$-dimensional.  This stack is naturally isomorphic to the stack quotient $\AA_{Q,\dd}/\Gl_{\dd}$.
\smallbreak
Let $N=\BZ^r$ be a lattice, and fix a function $\tau\colon Q_1\rightarrow N$.  We fix the torus $T=\Hom_{\mathrm{Grp}}(N,\BC^*)$.  The torus $T$ acts on $\AA_{Q,\dd}$ by sending
\[
(\mathbf{t},(\rho_a)_{a\in Q_1})\mapsto (\mathbf{t}(\tau(a))\cdot \rho_a)_{a\in Q_1},
\]
and this action commutes with the action of $\Gl_{\dd}$ by conjugation.  We denote by
\[
\Mst^T_{\dd}(Q)\cong \AA_{Q,\dd}/(\Gl_{\dd}\times T)
\]
the quotient of $\Mst_{\dd}(Q)$ by the $T$-action.  Precisely, a homomorphism $X\rightarrow \Mst^T_{\dd}(Q)$ is provided by a $T$-torsor $E\rightarrow X$, along with a flat family $\CF$ of $\dd$-dimensional $\BC Q$-modules on $E$, and a $T$-equivariant structure on $\mathcal{F}$, such that for every $a\in Q_1$ the following diagram of coherent sheaves on $T\times E$ commutes:
\[
\xymatrix{
l^*\CF\ar[d]^{\cong}\ar[rrr]^-{l^*\rho(a)}&&&l^*\CF\ar[d]^{\cong}\\
p^*\CF\ar[rrr]_-{\prod_{i\leq r}z_i^{\tau(a)_i}\cdot p^*\rho(a)}&&&p^*\CF.
}
\]  
Here, $l\colon T\times E\rightarrow E$ is the action, $p\colon T\times E\rightarrow E$ is the projection, $\rho(a)$ is the action of $a$ on $\CF$, $z_i$ are coordinates on $T$, and the vertical isomorphisms are defined by the $T$-equivariant structure on $\CF$.  
%We also explicitly describe the natural morphism
%\begin{equation}
%\label{shfn}
%s\colon\Mst^T_{\dd'}(Q')\times\Mst^T_{\dd''}(Q')\rightarrow \Mst^T_{\dd'}(Q')\times_{\B T}\Mst^T_{\dd''}(Q')
%\end{equation}
%where $\dd',\dd''\in\BN^{Q'_0}$.  If $p'\colon E'\rightarrow X$ and $p'\colon E''\rightarrow X$ are a pair of $T$-torsors on a scheme $X$, we start by giving the trivial vector bundle $\mathscr{O}_{E''}$ the trivial $T$-equivariant structure.  Then if $\CF$ is a $T$-equivariant coherent sheaf on $E'$, we first consider $\pi^*\CF$, where $\pi\colon E'\times_X E''\rightarrow E'$ is the projection.  This is naturally a $T\times T$-equivariant sheaf (via the trivial $T$-equivariant structure on $\mathscr{O}_{E''}$) that descends to a $T$-equivariant sheaf on $E'\times_T E''$ (along with the action of $\rho(a)$ for each $a\in Q_1$), which we denote $\CF\times_T E''$.  Then the morphism $s$ is given by
%\begin{equation}
%\label{shii}
%s((p'\colon E'\rightarrow X,\CF'),(p''\colon E''\rightarrow X,\CF''))=(E'\times_T E'',\CF'\times_T E'',\CF''\times_T E').
%\end{equation}
We will use the morphism
\begin{equation}
\label{lhfn}
\jmath\colon\Mst^T_{\dd'}(Q)\times_{\B T}\Mst^T_{\dd''}(Q)\rightarrow \Mst^T_{\dd'}(Q)\times\Mst^T_{\dd''}(Q)
\end{equation}
induced by the structure map $\B T\rightarrow \pt$.
\smallbreak
We denote by 
\[
\Msp_{\dd}(Q)=\Spec(\Gamma(\BA_{Q,\dd})^{\Gl_{\dd}})
\]
the coarse moduli space.  For $K\supset \BC$ a field, $K$-points of $\Msp_{\dd}(Q)$ correspond to isomorphism classes of semisimple $\dd$-dimensional $K Q$-modules.  Since the $T$-action commutes with the $\Gl_{\dd}$-action, $\Msp_{\dd}(Q)$ carries a natural $T$-action.  We define the stack-theoretic quotient
\[
\Msp^T_{\dd}(Q)\coloneqq \Msp_{\dd}(Q)/T.
\]
In the case of trivial $T$, there is a natural morphism
\[
\JH\colon \Mst_{\dd}(Q)\rightarrow \Msp_{\dd}(Q)
\]
which, at the level of points, takes modules to their semisimplifications, and at the level of algebraic geometry is just the affinization morphism.  This extends naturally to a morphism
\begin{equation}
\label{JHT_def}
\JH^T\colon \Mst_{\dd}^T(Q)\rightarrow \Msp^T_{\dd}(Q).
\end{equation}
Precisely, if the morphism $f\colon X\rightarrow \Mst^T_{\dd}(Q)$ is represented by the $T$-torsor $E\rightarrow X$ and the morphism $\tilde{f}\colon E\rightarrow \Mst_{\dd}(Q)$, the morphism $\JH^T\circ f$ is represented by the same torsor $E$, along with the morphism $\JH\circ \tilde{f}$.
\smallbreak
Given two dimension vectors $\dd',\dd''\in\BN^{Q_0}$ with $\dd=\dd'+\dd''$ we define $\AA_{Q,\dd',\dd''}\subset \AA_{Q,\dd}$ and $\Para_{\dd',\dd''}\subset \Gl_{\dd}$ to be the space of representations, and changes of basis, respectively, preserving the flags $0\subset \BC^{\dd'(i)}\subset \BC^{\dd(i)}$ for each $i\in Q_0$.  Denoting by $\Mst_{\dd',\dd''}(Q)$ the stack of short exact sequences $0\rightarrow \rho'\rightarrow \rho\rightarrow \rho''\rightarrow 0$ of $\BC Q$-modules, for which the dimension vectors of $\rho'$ and $\rho''$ are $\dd',\dd''$ respectively, there is a natural isomorphism
\[
\Mst_{\dd',\dd''}(Q)\cong\AA_{Q,\dd',\dd''}/\Gl_{\dd',\dd''}.
\]
We define $\Mst_{\dd',\dd''}^T(Q)\coloneqq\AA_{Q',\dd',\dd''}/(\Gl_{\dd',\dd''}\times T)$.

\subsection{Cohomological Hall algebras}
\label{CoHAsec}
Let $Q$ be a quiver.  Let $W\in\BC Q/[\BC Q,\BC Q]_{\mathrm{vect}}$ be a potential, which we assume to be $T$-invariant in the following sense: for every cyclic word $a_1\ldots a_c$ appearing in $W$
\begin{equation}
\label{WTI}
\sum_{i=1}^c\tau(a_i)=0.
\end{equation}
The potential $W$ induces a function $\Tr(W)$ on $\AA_{Q,\dd}$ which is $\Gl_{\dd}$-invariant by cyclic invariance of the trace map, and $T$-invariant by \eqref{WTI}.  We define
\begin{equation}
\label{undobj}
\CoHA^T_{Q,W}\coloneqq \bigoplus_{\dd\in\BN^{Q_0}}\HO(\Mst^T_{\dd}(Q),\phip{\Tr(W)})
\end{equation}
the underlying $\mathbb{N}^{Q_0}$-graded vector space of the critical cohomological Hall algebra, defined by Kontsevich and Soibelman in \cite{KS2}.  Consider the convolution diagram
\[
\xymatrix{
\Mst^T_{\dd'}(Q)\times_{\B T}\Mst^T_{\dd''}(Q)&\ar[l]_-{\pi_1\times \pi_3}\Mst^T_{\dd',\dd''}(Q)\ar[r]^-{\pi_2}& \Mst^T_{\dd}(Q)
}
\]
where $\pi_1,\pi_2,\pi_3$ take a flag $0\subset \CF'\subset \CF$ of flat families of $\BC Q$-modules to the families $\CF'$, $\CF$, $\CF/\CF'$ respectively.
We define 
\begin{align}
m'_{\dd',\dd''}\colon&\HO\left(\Mst^T_{\dd'}(Q)\times_{\B T}\Mst^T_{\dd''}(Q),\phip{\Tr(W)\boxplus\Tr(W)}\right)\rightarrow \HO\left(\Mst^T_{\dd}(Q),\phip{\Tr(W)}\right)\label{mdashdef}\\
&=\pi_{2,\star}\circ(\pi_1\times\pi_3)^{\star}\nonumber
\end{align}
and $m_{\dd',\dd''}=m'_{\dd',\dd''}\circ \jmath^{\star}\circ \TS$ where $\jmath$ is as in \eqref{lhfn} and $\TS$ is the Thom--Sebastiani isomorphism.  We then define the multiplication
\[
m\colon \CoHA^T_{Q,W}\otimes \CoHA^T_{Q,W}\rightarrow \CoHA^T_{Q,W}
\]
by summing $m_{\dd',\dd''}$ over all pairs of dimension vectors.

The multiplication $m$ makes $\CoHA^T_{Q,W}$ into a $\BN^{Q_0}$-graded associative algebra.  If we assume that $Q$ is \textit{symmetric}, meaning that for every pair of vertices $i,j\in Q_0$ there are as many arrows from $i$ to $j$ as there are from $j$ to $i$, then the morphism $m$ preserves the cohomological degree, and $\CoHA^T_{Q,W}$ is an associative algebra in the category of $\BN^{Q_0}\oplus\BZ$-graded vector spaces.
\smallbreak
Let $\mathcal{F}$ be a flat family of $\dd$-dimensional $\mathbb{C}Q$-modules on a scheme $X$.  For each $i\in Q_0$ we obtain a vector bundle $\CF_i=e_i\cdotsh \CF$ of rank $\dd_i$.  There is a determinant line bundle $\Det(\CF_i)$ on $X$, which is defined to be the top exterior power of the underlying bundle of $\CF_i$.  This defines a morphism of stacks for each $i\in Q_0$:
\[
\Det_{(i)}\colon \Mst^T_{\dd}(Q)\rightarrow \B \BC_u^*.
\]
Alternatively, we can construct this morphism as the composition
\begin{align*}
\Det_{(i)}=&\left(\Det\colon \B\Gl_{\dd_i}\rightarrow \B\BC^*_u\right)\circ \left(\pi_{\pt/\Gl_{\dd_i}}\colon\pt/(\Gl_{\dd_i}\times T)\rightarrow \B \Gl_{\dd_i}\right) \\&\circ \left(\AA_{Q,\dd}/(\Gl_{\dd}\times T)\rightarrow \pt/(\Gl_{\dd}\times T)\right).
\end{align*}
We denote by $u_{(i)}\in \HO(\Mst^T_{\dd}(Q),\BQ)$ the image of $u\in\HO_{\BC^*_u}$ under $\Det_{(i)}^{\star}$.  Consider the morphism
\begin{align*}
l_i=(\Det_{(i)}, \id)\colon &\Mst^T_{\dd}(Q)\rightarrow\B \BC_u^*\times  \Mst^T_{\dd}(Q).
\end{align*}
Since the function $\Tr(W)$ is pulled back from the function $0\boxplus\Tr(W)$ on the target of $l_i$ we obtain the morphism in vanishing cycle cohomology
\[
\HO(\B \BC^*_u,\BQ)\otimes \HO(\Mst^T_{\dd}(Q),\phip{\Tr(W)}) \xrightarrow{l_i^{\star}} \HO(\Mst^T_{\dd}(Q),\phip{\Tr(W)})
\]
which defines an action of the algebra $\mathbb{Q}[u]=\HO_{\BC^*_u}$ on $\CoHA^T_{Q,W}$, for each $i\in Q_0$: 
\begin{equation}
\label{lidef}
p(u)\bullet_i \alpha\coloneqq l_i^{\star}(p(u)\otimes \alpha).
\end{equation}
Similarly, we define the morphism
\begin{align}
\label{ldef}
l=(\bigotimes_{i\in Q_0}\Det_{(i)}, \id)\colon &\Mst^T_{\dd}(Q)\rightarrow\B \BC_u^*\times  \Mst^T_{\dd}(Q)
\end{align}
and an action of $\HO_{\BC_u^*}$ on $\CoHA^T_{Q,W}$ by 
\begin{equation}
\label{udef}
p(u)\bullet\alpha=l^{\star}(p(u)\otimes\alpha).
\end{equation}

\subsection{PBW isomorphism}
In this section we assume $T=\{1\}$ for simplicity, and also assume that $Q$ is symmetric.  It follows that there exists a (non-unique) bilinear form
\[
\psi\colon\mathbb{Z}^{Q_0}\times\mathbb{Z}^{Q_0}\rightarrow \mathbb{Z}^{Q_0}
\]
satisfying 
\begin{align*}
\psi(\dd',\dd'')+\psi(\dd'',\dd')=\chi_{Q}(\dd',\dd')\chi_{Q}(\dd'',\dd'')+\chi_{Q}(\dd',\dd'')&\quad\textrm{mod }2
\end{align*}
for all $\dd',\dd''$.  We fix a choice of such $\psi$ in what follows, and denote by $\CoHA^{\psi}_{Q,W}$ the cohomological Hall algebra $\CoHA_{Q,W}$ with the multiplication twisted by setting 
\[
m^{\TW}_{\dd',\dd''}=(-1)^{\psi(\dd',\dd'')}m_{\dd',\dd''}
\]
where $m_{\dd',\dd''}$ and $m^{\TW}_{\dd',\dd''}$ are the restrictions of the products in $\CoHA_{Q,W}$ (resp. $\CoHA^{\psi}_{Q,W}$) to the pieces of degree $\dd',\dd''$.  In our main application, we will have that $Q=Q^{(3)}$, so that $\chi_{Q}$ is even, and we can (and will) set $\psi=0$.  
\begin{remark}
In what follows, if we write that something is true of $\CoHA^{(\psi)}_{Q,W}$, we mean that it is true of both $\CoHA^{\psi}_{Q,W}$ and $\CoHA_{Q,W}$.  Note that the underlying graded vector spaces of these algebras are identical.
\end{remark}

By \cite{QEAs} the direct image $\JH_*\phip{\Tr(W)}\BQ_{\Mst_{\dd}(Q)}^{\vir}$ splits: there is an isomorphism in $\Dub^+(\Perv(\Msp_{\dd}(Q)))$:
\begin{equation}
\label{splitting}
\JH_*\phip{\Tr(W)}\BQ_{\Mst_{\dd}(Q)}^{\vir}\cong \bigoplus_{i\geq 1}{}^{\fp}\!\Ho^i(\JH_*\phip{\Tr(W)}\BQ_{\Mst_{\dd}(Q)}^{\vir})[-i].
\end{equation}
There is a natural morphism
\begin{equation}
\label{PF_def}
\FP_i\CoHA^{(\psi)}_{Q,W,\dd}\coloneqq \HO(\Msp_{\dd}(Q),{}^{\fp}\tau^{\leq i}\JH_*\phip{\Tr(W)}\BQ_{\Mst_{\dd}(Q)}^{\vir})\rightarrow \CoHA^{(\psi)}_{Q,W,\dd}
\end{equation}
induced by the natural transformation ${}^{\fp}\tau^{\leq i}\rightarrow \id$ and it follows from the splitting \eqref{splitting} that \eqref{PF_def} is injective, and the subspaces $\FP_i\CoHA^{(\psi)}_{Q,W,\dd}$ provide a filtration of $\CoHA^{(\psi)}_{Q,W,\dd}$, called the \textit{perverse filtration}.  
\smallbreak
By \cite[Sec.5.3]{QEAs} the multiplication on $\CoHA^{(\psi)}_{Q,W}$ respects the perverse filtration.  There is vanishing
\begin{equation}
\label{Pvanishing}
\FP_0\!\CoHA^{(\psi)}_{Q,W}=0,
\end{equation}
and furthermore the subspace 
\[
\fg_{Q,W}\coloneqq \FP_1\!\CoHA^{\psi}_{Q,W}\subset\CoHA^{\psi}_{Q,W}
\]
is preserved by the commutator Lie bracket $[-,-]$ by \cite[Cor.6.11]{QEAs} (here is where it starts to become necessary to choose the $\psi$-twisted multiplication).  There are isomorphisms
\[
\fg_{Q,W,\dd}=\HO(\Msp_{\dd}(Q),\DTS_{Q,W,\dd})[-1]
\]
where $\DTS_{Q,W,\dd}=({}^{\fp}\tau^{\leq 1}\JH_*\phip{\Tr(W)}\BQ_{\Mst_{\dd}(Q)}^{\vir})[1]$ is the \textit{BPS sheaf} defined in \cite[Thm.A]{QEAs}.  This is a perverse sheaf, with a natural lift to the category of mixed Hodge modules on $\Msp_{\dd}(Q)$, so that $\fg_{Q,W}$ carries a natural mixed Hodge structure, which is respected by the Lie algebra structure on it; we refer to \cite[Sec.2]{QEAs} for details.  We call $\fg_{Q,W}$ the \textit{BPS Lie algebra} associated to $(Q,W)$.  It inherits a cohomological grading from $\CoHA^{\psi}_{Q,W}$, and for $i\in \BZ$ we denote by $\fg^i_{Q,W}$ the summand of cohomological degree $i$.
\begin{theorem}\cite[Thm.C]{QEAs}
\label{IT}
The PBW morphism
\begin{equation}
\label{PBWiso}
\Psi\colon\Sym(\fg_{Q,W}\otimes \HO_{\BC^*_u})\rightarrow \CoHA^{\psi}_{Q,W}
\end{equation}
is an isomorphism of cohomologically graded $\BN^{Q_0}$-graded vector spaces.
\end{theorem}
The morphism $\Psi$ is constructed as follows.  Firstly, we consider the composition
\[
\Psi'\colon\HO(\B \BC^*_u,\BQ)\otimes \fg_{Q,W}\hookrightarrow \HO(\B \BC^*_u,\BQ)\otimes\CoHA^{\psi}_{Q,W}\xrightarrow{l^{\star}}\CoHA^{\psi}_{Q,W}
\]
with $l$ as defined in \eqref{ldef}.  Since the target carries an algebra structure, $\Psi'$ extends uniquely to a morphism of algebras
\[
\Phi\colon\mathrm{T}\left(\HO(\B \BC^*_u,\BQ)\otimes \fg_{Q,W}\right)\rightarrow \CoHA^{\psi}_{Q,W}
\]
from the free associative algebra generated by the domain of $\Psi'$.  Then $\Psi$ is defined to be the restriction of $\Phi$ to the subspace of (supersymmetric) tensors.  Note that the morphism $\Psi$ will typically not be a morphism of algebras; the CoHA $\CoHA^{\psi}_{Q,W}$ will in general have an interesting algebra structure, and not be a free (super)commutative algebra.

\subsection{Deformed BPS Lie algebras}
Let $Q$ be a quiver, which we continue to assume is symmetric.  Let $T$ be a torus acting on the arrows of $Q$, and let $W\in\BC Q/[\BC Q,\BC Q]_{\mathrm{vect}}$ be a $T$-invariant potential.  Recall the map $\JH^T$ from \eqref{JHT_def}.  We define\footnote{The perverse truncation functors are modified according to the conventions introduced in \S \ref{toolkit}.} the $\HO_{T}$-module
\[
\fg^T_{Q,W,\dd}\coloneqq \HO(\Msp^T_{\dd}(Q),{}^{\fp'}\!\!\tau^{\leq 1}\JH_*^T\phip{\Tr(W)}\BQ_{\Mst^T_{\dd}(Q)}^{\vir}).
\]
Again, this cohomologically graded vector space admits a natural lift to the category of cohomologically graded mixed Hodge structures, for which we refer to \cite[Sec.2]{QEAs} for details.  Applying the derived global sections functor to the natural morphism of complexes ${}^{\fp'}\!\!\tau^{\leq 1}\JH_*^T\phip{\Tr(W)}\BQ_{\Mst^T_{\dd}(Q)}^{\vir}\rightarrow \JH_*^T\phip{\Tr(W)}\BQ_{\Mst^T_{\dd}(Q)}^{\vir}$ we obtain the morphism
\[
\iota\colon \fg^T_{Q,W,\dd}\rightarrow \CoHA^{T,\psi}_{Q,W,\dd}.
\]
\begin{theorem}
\label{Tinj}
The morphism $\iota$ is injective, so that there is a natural inclusion of $\mathbb{N}^{Q_0}$-graded subspaces $\fg^T_{Q,W}\subset \CoHA_{Q,W}^{T,\psi}$.  The subspace $\fg^T_{Q,W}$ is closed under the commutator Lie bracket on $\CoHA_{Q,W}^{T,\psi}$.  Assume, furthermore, that $\fg_{Q,W}^T$ is pure, as a mixed Hodge structure.  Then there is an isomorphism of vector spaces 
\begin{equation}
\label{BPSD}
\fg^T_{Q,W}\cong \fg_{Q,W}\otimes \HO_T,
\end{equation}
and the PBW morphism
\[
\Sym_{\HO_T}(\fg_{Q,W}^T\otimes\HO_{\BC^*_u})\rightarrow \CoHA_{Q,W}^{T,\psi}
\]
constructed from the Hall algebra product on the target is an isomorphism.
\end{theorem}
\begin{proof}
The proof is a minor modification of the analogous statements from \cite{QEAs}, to which we refer for more details.  Firstly, we claim that $\JH^T$ is approximated by projective morphisms in the sense of \cite[Sec.4.1]{QEAs}.  For this the proof is as in \cite{QEAs}, using the quotient of moduli spaces of framed $\BC Q$-modules by the $T$-action as auxiliary spaces.  From this it follows as in \cite[Prop.4.4]{QEAs} that there is a splitting
\[
\HO(\Msp^T_{\dd}(Q),\JH^T_*\phip{\Tr(W)}\BQ_{\Mst^T_{\dd}(Q)}^{\vir})\cong \bigoplus_{i\in \BZ}\HO(\Msp^T_{\dd}(Q),{}^{\fp}\!\Ho^i\!\left(\JH^T_*\phip{\Tr(W)}\BQ_{\Mst^T_{\dd}(Q)}^{\vir}\right)[-i])
\]
and the morphism
\[
\HO(\Msp^T_{\dd}(Q),{}^{\fp'}\!\!\tau^{\leq 1}\JH^T_*\phip{\Tr(W)}\BQ_{\Mst^T_{\dd}(Q)}^{\vir})\rightarrow \HO(\Msp^T_{\dd}(Q),\JH^T_*\phip{\Tr(W)}\BQ_{\Mst^T_{\dd}(Q)}^{\vir})
\]
has a left inverse, and the injectivity part of the theorem follows.  
\smallbreak
The monoid morphism $\Msp(Q)\times \Msp(Q)\rightarrow \Msp(Q)$ taking a pair of semisimple $\BC Q$-modules to their direct sum is $T$-equivariant, so that we can define the monoid structure 
\begin{equation}
\label{oplusdef}
\oplus\colon \Msp^T(Q)\times_{\B T}\Msp^T(Q)\rightarrow \Msp^T(Q)
\end{equation}
by passing to the stack-theoretic quotient by the $T$-action.  Given a pair of complexes of sheaves $\CF$ and $\CG$ on two stacks $\FX$ and $\FY$ over a stack $\FZ$ we define the external product, a complex of sheaves on $\FX\times_{\FZ}\FY$ in the usual way:
\[
\CF\boxtimes_{\FZ}\CG\coloneqq \pi_{\FX}^*\CF\otimes \pi_{\FY}^*\CG.
\]
This defines a symmetric tensor product on complexes of sheaves on $\Msp^T(Q)$: we define
\[
\CF\boxtimes_{\oplus}\CG=\oplus_*(\CF\boxtimes_{\B T}\CG).
\]
Let $p\colon \Msp^T(Q)\rightarrow \B T$ be the canonical map.  This is a morphism of monoid objects over $\B T$, where $\B T$ is considered as a monoid over itself via the identity morphism.  For arbitrary complexes of sheaves on $\Msp^T(Q)$ it is too much to hope that there is an isomorphism 
\begin{equation}
\label{tmthp}
\HO(\Msp^T(Q),\CF\boxtimes_{\oplus}\CG)\cong \HO(\Msp^T(Q),\CF)\otimes_{\HO_T}\HO(\Msp^T(Q),\CG).
\end{equation}
However, let us assume that the complexes $p_*\CF$ and $p_*\CG$ lift to pure complexes of mixed Hodge modules.  Then via the usual spectral sequence argument (see \cite[Sec.9]{preproj}), there are isomorphisms 
\begin{align*}
\HO(\Msp^T(Q),\CF)\cong &\HO(\Msp(Q),\overline{\CF})\otimes \HO_T\\
\HO(\Msp^T(Q),\CG)\cong &\HO(\Msp(Q),\overline{\CG})\otimes \HO_T\\
\HO(\Msp^T(Q),\CF\boxtimes_{\oplus}\CG)\cong &(\HO(\Msp(Q),\overline{\CF})\otimes \HO(\Msp(Q),\overline{\CG}))\otimes \HO_T
\end{align*}
where $\overline{\CF}$ and $\overline{\CG}$ denote the inverse images of $\CF$ and $\CG$ along the canonical morphism $\Msp(Q)\rightarrow \Msp^T(Q)$, and a \textit{natural} isomorphism \eqref{tmthp}.
\smallbreak
Via the same morphisms as in \cite[Sec.5.1]{QEAs} (see also \S \ref{CoHAsec}), the complex $\JH_*^T\phip{\Tr(W)}\BQ_{\Mst^T(Q)}^{\vir}$ carries a relative Hall algebra structure, i.e. for $\dd=\dd'+\dd''$ there is a morphism
\[
\JH^T_*\phip{\Tr(W)}\BQ_{\Mst^T_{\dd'}(Q)}^{\vir}\boxtimes_{\oplus}\JH^T_*\phip{\Tr(W)}\BQ_{\Mst^T_{\dd''}(Q)}^{\vir}\rightarrow \JH^T_*\phip{\Tr(W)}\BQ_{\Mst^T_{\dd}(Q)}^{\vir}
\]
satisfying the obvious associativity condition.  We claim that
\begin{itemize}
\item
The perverse truncation ${}^{\fp'}\!\!\tau^{\leq 0}\JH^T_*\phip{\Tr(W)}\BQ_{\Mst^T_{\dd}(Q)}^{\vir}$ vanishes.
\item
The morphism 
\[
{}^{\fp'}\!\!\tau^{\leq 1}\JH^T_*\phip{\Tr(W)}\BQ_{\Mst^T_{\dd}(Q)}^{\vir}[1]\boxtimes_{\oplus}{}^{\fp'}\!\!\tau^{\leq 1}\JH^T_*\phip{\Tr(W)}\BQ_{\Mst^T_{\dd}(Q)}^{\vir}[1]\rightarrow {}^{\fp'}\!\!\Ho^2\!\left(\JH^T_*\phip{\Tr(W)}\BQ_{\Mst^T_{\dd}(Q)}^{\vir}\right)
\]
given by restricting the commutator bracket on the relative CoHA $\JH^T_*\phip{\Tr(W)}\BQ_{\Mst^T_{\dd}(Q)}^{\vir}$ (with sign twisted by $\psi$) vanishes.
\end{itemize}
Taking derived global sections, these two claims prove that $\fg^T_{Q,W}$ is closed under the commutator Lie bracket.  The first claim concerns the vanishing of certain perverse sheaves on $\Msp^T(Q)$, and the second concerns the vanishing of certain morphisms of perverse sheaves on $\Msp^T(Q)$.  It follows from \cite[Sec.4]{BBD} that the forgetful functor $\Perv(\Msp^T(Q))\rightarrow \Perv(\Msp(Q))$ is faithful, so both these claims reduce to the analogous claims on $\Msp(Q)$, which are proved as \cite[Thm.A]{QEAs}, \cite[Cor.6.11]{QEAs} respectively.
\smallbreak
The isomorphism \eqref{BPSD} follows from the usual spectral sequence argument (as in \cite[Sec.9.2]{preproj}).  Since $p$ is a morphism of monoids over $\B T$, the complex $p_*\JH^T_*\phip{\Tr(W)}\BQ_{\Mst^T(Q)}^{\vir}$ inherits a ($\psi$-twisted) Hall algebra structure, along with an action of $\HO_{\BC^*_u}$.  We claim that the morphism of complexes on $\B T$
\[
\Sym_{\boxtimes_{\B T}}\left(p_*{}^{\fp'}\!\!\tau^{\leq 1}\JH^T_*\phip{\Tr(W)}\BQ_{\Mst^T(Q)}^{\vir}\otimes\HO_{\BC^*_u}\right)\rightarrow p_*\JH^T\phip{\Tr(W)}\BQ_{\Mst^T(Q)}^{\vir}
\]
is an isomorphism.  This may be checked one perverse degree at a time; then using the fact that the forgetful morphism $\Perv^T(\pt)\rightarrow \Perv(\pt)$ is faithful, the statement reduces to the non-equivariant PBW theorem (Theorem \ref{IT}).  The equivariant PBW theorem then follows.
\end{proof}
We refer to $\fg^T_{Q,W}$ as the \textit{deformed} BPS Lie algebra.  By the above theorem, as long as $\fg_{Q,W}$ is pure, this deformed Lie algebra provides a flat family of Lie algebras over $\mathfrak{t}\cong \mathrm{\Spec}(\HO_T)$ specialising to the original BPS Lie algebra $\fg_{Q,W}$ at the origin.

\subsection{Kac polynomials and Kac--Moody Lie algebras}
\label{Kac_sec}
Given a quiver $Q$ without loops we denote by $\fn^-_Q$ the negative half of the associated Kac--Moody Lie algebra.  This is the free $\BN^{Q_0}$-graded Lie algebra with generators $F_i$ in $\mathbb{N}^{Q_0}$-degree $1_i$, subject to the Serre relations
\[
[F_i,-]^{a_{ij}+1}(F_j)=0
\]
for $i\neq j$.  Here $a_{ij}$ denotes the number of edges joining $i$ and $j$ in the graph obtained from $Q$ by forgetting orientations of edges.  If the underlying graph of $Q$ is a type ADE Dynkin diagram, $\fn^-_Q$ is (one half of) the associated simple Lie algebra.%  If the underlying graph is an extended type ADE Dynkin diagram, $\fn^-_Q$ is the (one half of) affinized Lie algebra.  In this case we let $\delta\in \BN^{Q_0}$ be the dimension vector of the primitive imaginary root.  For instance in type A we have $\delta=(1,\ldots,1)$.
\smallbreak
Given a quiver $Q$ and a field $K$, a $KQ$-module $\rho$ is called indecomposable if it cannot be written as $\rho=\rho'\oplus\rho''$ with $\rho',\rho''\neq 0$.  It is called absolutely indecomposable if $\rho\otimes_{K}\overline{K}$ is an indecomposable $\overline{K}Q$-module.  For $q=p^n$ a prime power, let $\kac_{Q,\dd}(q)$ denote the number of isomorphism classes of absolutely indecomposable $\dd$-dimensional $\BF_q Q$-modules.  Victor Kac showed \cite{Kac83} that $\kac_{Q,\dd}(q)$ is a polynomial in $q$, and moreover an element of $\BZ[q]$.
\smallbreak
If $Q$ is an orientation of a finite ADE graph, then 
\begin{equation}
\label{Kac_ADE}
\kac_{Q,\dd}(q)=\dim(\fn^-_{Q,\dd}).
\end{equation}
In particular, it is constant.  Conversely, if $Q$ is any other type of quiver, $\kac_{Q,\dd}(q)$ is nonconstant for at least some values of $\dd$.
%  If $Q$ is an orientation of an extended Dynkin graph, then
%\begin{equation}
%\kac_{Q,\dd}(q)=\begin{cases}
%\dim(\fn^-_{Q})+q&\textrm{if }\dd\in \BZ_{\geq 0}\delta\\
%\dim(\fn^-_Q)&\textrm{otherwise}.
%\end{cases}
%\end{equation}

\subsection{Preprojective CoHAs}
We consider the class of special cases of quiver with potential $(Q,W)=(\tilde{Q'},\tilde{W})$ for a finite quiver $Q'$.  Although we do not assume that $Q'$ is symmetric, $Q$ obviously is.  We call these CoHAs \textit{preprojective CoHAs}, for reasons that will become evident in \S \ref{dimred_sec}.  By \cite[Thm.4.7]{preproj} there is an equality of generating series
\begin{equation}
\label{KactoDT}
\sum_{i\in \BZ}\dim(\fg_{Q,\tilde{W},\dd}^i)q^{i/2}=\kac_{Q',\dd}(q^{-1}).
\end{equation}
where $\kac_{Q',\dd}(q)$ is the Kac polynomial of $Q'$.  So in particular, $\fg_{Q,\tilde{W}}$ is concentrated in even, negative cohomological degrees.  
\begin{proposition}\cite[Thm.6.6]{preproj3}
\label{dzprop}
Let $Q'$ be a finite quiver, and denote by $Q^{\mathrm{re}}$ the quiver obtained from $Q'$ by removing all vertices that support a loop, along with any arrows beginning or ending at such a vertex.  Set $Q=\tilde{Q'}$.  There is an isomorphism of $\mathbb{N}^{Q_0}$-graded Lie algebras
\begin{equation}
\label{dzpart}
\fg_{Q,\tilde{W}}^0\cong \fn^-_{Q^{\mathrm{re}}}.
\end{equation}
\end{proposition}
The Lie algebra on the right hand side a priori carries a $\BN^{Q^{\mathrm{re}}_0}$-grading.  We promote this to a $\mathbb{N}^{Q_0}$-grading via extension by zero along the embedding $\BN^{Q^{\mathrm{re}}_0}\subset\mathbb{N}^{Q_0}$ induced by the inclusion $Q^{\mathrm{re}}_0\subset Q_0$. 

Choose a torus $T=(\BC^*)^l$, and a weighting function $\tau\colon Q_1\rightarrow \BZ^l$ defining an action of $T$ on the category of $\BC Q$-modules, for which $\tilde{W}$ is invariant.  The following theorem is a consequence of the \textit{purity} of the natural mixed Hodge structure on $\CoHA_{Q,\tilde{W}}$, see \cite{preproj} for details.
\begin{proposition}\cite[Thm.9.6]{preproj}
\label{pur_form}
Let $l\colon \BZ^r\rightarrow \BZ^{r'}$ be a split surjection of lattices, with complement $\BZ^{r''}$, inducing the split inclusion of tori
\[
T'=\Hom_{\mathrm{Grp}}(\BZ^{r'},\BC^*)\hookrightarrow T=\Hom_{\mathrm{Grp}}(\BZ^{r},\BC^*)
\]
with complement $T''=\Hom_{\mathrm{Grp}}(\BZ^{r''},\BC^*)$.  There is an isomorphism of $\BN^{Q_0}$-graded, cohomologically graded $\HO_T$-modules\footnote{We omit the $\psi$ superscript because the $\psi$-twist does not feature in the $\HO_T$-module structure.}
\begin{equation}
\label{diso}
\CoHA^T_{Q,\tilde{W}}\cong \CoHA^{T'}_{Q,\tilde{W}}\otimes\HO_{T''}
\end{equation}
and an isomorphism of algebras
\begin{equation}
\label{algeiso}
\CoHA^{T',(\psi)}_{Q,\tilde{W}}\cong \CoHA^{T,(\psi)}_{Q,\tilde{W}}\otimes_{\HO_T}\HO_{T'}.
\end{equation}
\end{proposition}
For the special case $r'=0$, the above proposition yields an isomorphism of $\HO_T$-modules
\begin{equation}
\label{TZ}
\CoHA^{T}_{Q,\tilde{W}}\cong \CoHA_{Q,\tilde{W}}\otimes\HO_{T}
\end{equation}
and an isomorphism of algebras
\[
\CoHA^{(\psi)}_{Q,\tilde{W}}\cong \CoHA^{T,(\psi)}_{Q,\tilde{W}}/\fm_T\cdot \CoHA^{T,(\psi)}_{Q,\tilde{W}},
\]
where we denote by $\fm_T$ the maximal homogeneous ideal in $\HO_T$.   We will be particularly interested in the action of the torus $T=(\BC^*)^2=\Hom_{\mathrm{Grp}}(\BZ^2,\BC^*)$ acting via the weighting function
\begin{align}
\label{ttaudef}
\tilde{\tau}\colon &Q_1\rightarrow \BZ^2\\
&a\mapsto (1,0);\quad\quad
a^*\mapsto (0,1);\quad\quad
\omega_i\mapsto (-1,-1).\nonumber
\end{align}

\subsection{Dimensional reduction}
\label{dimred_sec}
We recall from \cite[Appendix.A]{Da13} the following result.
\begin{theorem}\cite[Thm.A.1]{Da13}
Let $X=Y\times \BA^n$ be a smooth variety, let $f$ be a function on $X$ of weight one for the scaling action on $\BA^n$, so that we can write 
\[
f=\sum_{i=1}^nx_if_i
\]
for $x_1,\ldots,x_n$ a system of coordinates on $\BA^n$ and $f_1,\ldots,f_n$ functions on $Y$.  Set $Z=Z(f_1,\ldots,f_n)$, let $i\colon Z\hookrightarrow Y$ be the inclusion, and let $\pi\colon X\rightarrow Y$ be the projection.  Then there is a natural isomorphism
\[
i_*i^!\CF\rightarrow \pi_*\phip{f}\pi^*\CF.
\]
\end{theorem}
Applying the result to $\BD\BQ_{Y}$ and passing to global sections yields the isomorphism
\[
\HO^{\BoMo}(Z,\BQ)\rightarrow \HO(X,\phip{f}\BQ)[2\dim(Y)].
\]
The theorem extends in the obvious way to affine fibrations over stacks \cite[Cor.A.9]{Da13}.  In particular, for a quiver $Q$ (not assumed symmetric) and the affine fibration
\[
\Mst_{\dd}(\tilde{Q})\rightarrow \Mst_{\dd}(\overline{Q})
\]
the function $\Tr(\tilde{W})$ has weight one with respect to the action that scales the fibres (i.e. the space of choices of matrices for each of the loops $\omega_i$).  As such, we obtain an isomorphism
\[
\Psi_{\dd}\colon \HO^{\BoMo}(\Mst^T_{\dd}(\Pi_Q),\BQ)[-2\chi_Q(\dd,\dd)]\cong\HO(\Mst^T_{\dd}(\tilde{Q}),\phip{\Tr(\tilde{W})}).
\]
We slightly twist this isomorphism\footnote{See \cite[Lem.4.1]{RS17} for the origin of this sign.} by setting $\Psi'_{\dd}=(-1)^{\binom{\lvert\dd\lvert}{2}}\Psi_{\dd}$.  We use these sign-twisted isomorphisms, and the algebra structure on $\CoHA^{T,\chi}_{\tilde{Q},\tilde{W}}$, to induce an algebra structure on
\[
\CoHA^T_{\Pi_Q}\coloneqq \bigoplus_{\dd\in\BN^{Q_0}}\HO^{\BoMo}(\Mst^T_{\dd}(\Pi_Q),\BQ^{\vir}).
\]
Here, in the definition of the sign twist in the definition of $\CoHA^{T,\chi}_{\tilde{Q},\tilde{W}}$, we have chosen $\psi=\chi=\chi_Q$, noting that
\begin{align*}
\chi_Q(\dd',\dd'')+\chi_Q(\dd'',\dd')=\chi_{\tilde{Q}}(\dd',\dd')\chi_{\tilde{Q}}(\dd'',\dd'')+\chi_{\tilde{Q}}(\dd',\dd'')& \quad\textrm{mod }2.
\end{align*}
Recall that we have defined the complex $\BQ^{\vir}$ on the quotient of a smooth stack by a $T$-action to be the constant sheaf, shifted on each component of the stack by the dimension of the stack.  Since $\Mst_{\dd}(\Pi_Q)$ is not smooth, we need a new convention to make sense of the above definition of $\CoHA^T_{\Pi_Q}$: on $\Mst_{\dd}^{T}(\Pi_Q)$ we define $\BQ^{\vir}\coloneqq \BQ[-2\chi(\dd,\dd)]$.  
\smallbreak
The algebra structure on $\CoHA^T_{\Pi_Q}$ induced by the algebra structure on $\CoHA^{T,\chi}_{\tilde{Q},\tilde{W}}$ and the isomorphism $\Psi'$ is the same as the one defined by Schiffmann and Vasserot in \cite{ScVa13}, see \cite[Appendix]{RS17}, \cite{YZ16} for proofs.  Our algebra structure is thus isomorphic to theirs, with the added sign twist provided by $\chi_Q$.

\subsection{Affinizing BPS Lie algebras}
\label{CopSec}
For pairs $(\tilde{Q},\tilde{W})$ of a tripled quiver with potential there is extra structure on the CoHA $\CoHA^{\chi}_{\tilde{Q},\tilde{W}}$, ultimately derived from a factorization structure on the stack of representations of the Jacobi algebra.  We briefly describe this here; full details for more general 3-Calabi--Yau completions of 2-Calabi--Yau categories will appear in joint work with Hennecart, Kinjo, Schiffmann and Vasserot.  In the special case in which $Q$ is the one-loop quiver, i.e. the case we need for the main results of this paper, this structure can be extracted from \cite{KV19} via dimensional reduction.
\smallbreak
Under the isomorphism $j$ from \eqref{TJA} the element $j(\omega)$ is central in $\Jac(\tilde{Q},\tilde{W})$, so that every $\Jac(\tilde{Q},\tilde{W})$-module $\rho$ admits a canonical decomposition as a $\Jac(\tilde{Q},\tilde{W})$-module
\[
\rho\cong \bigoplus_{\lambda\in\BC}\rho_{\lambda}
\]
where $j(\omega)$ acts on $\rho_{\lambda}$ with unique generalised eigenvalue $\lambda$.  From this one deduces (see \cite[Lem.4.1]{preproj}) that for $U\subset \BC$ an open ball, if one denotes by $\Mst_{\dd}^U(\tilde{Q})$ the open substack satisfying the condition that each $\omega_i$ acts with generalised eigenvalues inside $U$, the restriction
\[
\HO(\Mst_{\dd}(\tilde{Q}),\phip{\Tr(\tilde{W})})\rightarrow \HO(\Mst^U_{\dd}(\tilde{Q}),\phip{\Tr(\tilde{W})})
\]
is an isomorphism.  Let $U'$ and $U''$ be disjoint open balls in $\BC$, then via the Thom--Sebastiani isomorphism one constructs the composition of morphisms
\begin{align*}
\Delta\colon \HO(\Mst(\tilde{Q}),\phip{\Tr(\tilde{W})}\BQ^{\vir})\rightarrow &\HO(\Mst^{U'\coprod U''}(\tilde{Q}),\phip{\Tr(\tilde{W})}\BQ^{\vir})\\&\cong\HO(\Mst^{U'}(\tilde{Q}),\phip{\Tr(\tilde{W})}\BQ^{\vir})\otimes \HO(\Mst^{U''}(\tilde{Q}),\phip{\Tr(\tilde{W})}\BQ^{\vir})\\&\cong \CoHA^{\chi}_{\tilde{Q},\tilde{W}}\otimes\CoHA^{\chi}_{\tilde{Q},\tilde{W}}
\end{align*}
defining a coproduct on $\CoHA^{\chi}_{\tilde{Q},\tilde{W}}$.  The second isomorphism is not quite canonical: there is a choice of sign, with the geometric origin explained in \cite[Lem.4.1]{RS17}.  It is almost cocommutative since the space of choices of $U'$ and $U''$ is connected, so that we can continuously swap them.  Because of the signs appearing in dimensional reduction, in order for this coproduct to be actually cocommutative, we again multiply the component $\Delta_{\dd',\dd''}$ by $(-1)^{\chi_Q(\dd',\dd'')}$.  The resulting coproduct is compatible with $m^{\TW}$, since $m^{\TW}$ lifts to an algebra structure on the object 
\[
(\Mst(\tilde{Q})\xrightarrow{\lambda} \Sym(\BA^1))_*\phip{\Tr(\tilde{W})}\BQ^{\vir}
\]
where $\lambda$ is the morphism recording the generalised eigenvalues of $j(\omega)$.  It follows by the Milnor--Moore theorem that $\CoHA_{\tilde{Q},\tilde{W}}^{\chi}$ is a universal enveloping algebra.  Furthermore, by the support lemma \cite[Lem.4.1]{preproj} $\lambda_*\DTS_{\tilde{Q},\tilde{W}}\otimes\HO_{\BC^*_u}$ is supported on the subspace of $\Msp(\tilde{Q})$ for which all of the generalised eigenvalues of $j(\omega)$ are identical, and so 
\[
\hat{\fg}_{\tilde{Q},\tilde{W}}\coloneqq \fg_{\tilde{Q},\tilde{W}}\otimes\HO_{\BC^*_u}\subset \CoHA^{\chi}_{\tilde{Q},\tilde{W}}
\]
is primitive for the coproduct $\Delta$ (by its construction), i.e. for $\alpha\in\hat{\fg}_{\tilde{Q},\tilde{W}}$, $\Delta(\alpha)=1\otimes \alpha+\alpha\otimes 1$.  By Theorem \ref{IT} and the PBW theorem for $\CoHA^{\chi}_{\tilde{Q},\tilde{W}}$ the subspace $\hat{\fg}_{\tilde{Q},\tilde{W}}$ has the same graded dimensions as the space of primitive elements for $\CoHA^{\chi}_{\tilde{Q},\tilde{W}}$.  Since all elements of $\hat{\fg}_{\tilde{Q},\tilde{W}}$ are primitive, we find
\begin{proposition}
\label{UEA_prop}
The subspace $\hat{\fg}_{\tilde{Q},\tilde{W}}\subset \CoHA^{\chi}_{\tilde{Q},\tilde{W}}$ is closed under the commutator Lie bracket, and $\CoHA^{\chi}_{\tilde{Q},\tilde{W}}=\UEA(\hat{\fg}_{\tilde{Q},\tilde{W}})$.
\end{proposition}
\begin{definition}
We call $\hat{\fg}_{\tilde{Q},\tilde{W}}$ the \textit{affine BPS Lie algebra} for $(\tilde{Q},\tilde{W})$.
\end{definition}
\begin{remark}
So far there are no examples in the literature of quivers with potential $\CoHA^{\psi}_{Q',W}$ for which $\fg_{Q',W}\otimes\HO_{\BC^*}$ has been shown \textit{not} to be closed under the commutator Lie bracket.  It would be interesting to know what class of BPS Lie algebras can be affinized this way, and in particular whether the ``partially fermionized'' cousins \cite{Da22BF} of the BPS Lie algebras $\fg_{\tilde{Q},\tilde{W}}$ belong to this class.
\end{remark}
Let $\BC^*$ act on $\BC \tilde{Q}$-modules via the weighting function $\tau\colon \tilde{Q}_1\rightarrow \mathbb{Z}$ which sends $a$ to $1$ and $a^*$ to $-1$ for all $a\in Q_1$, and $\omega_i$ to $0$ for all $i\in Q_0$.  Then the morphism $\lambda$ is $\BC^*$-invariant, so that we can repeat the construction of the commutative coproduct above for the CoHA
\[
\CoHA^{\BC^*,\chi}_{\tilde{Q},\tilde{W}},
\]
considered as an algebra object in the category of $\HO_{\mathbb{C}^*}$-modules.  
\begin{remark}
Via the symmetry that permutes the arrows $a,a^*,\omega$, the algebra $\CoHA^{\BC^*,\chi}_{Q^{(3)},\tilde{W}}$ carries a cocommutative coproduct in the same category, i.e. a morphism 
\[
\CoHA^{\BC^*}_{Q^{(3)},\tilde{W}}\rightarrow \CoHA^{\BC^*}_{Q^{(3)},\tilde{W}}\otimes_{\HO_{\mathbb{C}^*}}\CoHA^{\BC^*,\chi}_{Q^{(3)},\tilde{W}},
\]
and is a universal enveloping algebra, as long as $\BC^*$ acts on one of $a,a^*,\omega$ with weight zero.  
\end{remark}
Returning to the case of general tripled quivers, since the coproduct on $\CoHA^{\BC^*,\chi}_{\tilde{Q},\tilde{W}}$ is $\HO_{\BC^*}$-linear by construction, we obtain a $\HO_{\BC^*}$-linear Lie algebra of primitive elements $\hat{\fg}^{\BC^*}_{\tilde{Q},\tilde{W}}$:
\begin{definition}
We denote by $\hat{\fg}^{\BC^*}_{\tilde{Q},\tilde{W}}\subset \CoHA^{\BC^*,\chi}_{\tilde{Q},\tilde{W}}$ the Lie algebra of primitives, which we call the \textit{deformed affinized BPS Lie algebra} for $(\tilde{Q},\tilde{W})$.
\end{definition}
The following is a basic application of the Milnor--Moore theorem:
\begin{proposition}
\label{TPBW}
There is an isomorphism of $\HO_{\BC^*}$-algebra objects:
\begin{equation}
\label{dmm}
\UEA_{\HO_{\BC^*}}\!(\hat{\fg}^{\BC^*}_{\tilde{Q},\tilde{W}})\rightarrow \CoHA^{\BC^*,\chi}_{\tilde{Q},\tilde{W}}.
\end{equation}
\end{proposition}
The ``deformed'' in the definition is justified by the following
\begin{proposition}
\label{Tg}
There is an isomorphism of $\BN^{Q_0}$-graded, cohomologically graded $\HO_{\BC^*}$-modules.
\begin{equation}
\label{dgiso}
\hat{\fg}^{\BC^*}_{\tilde{Q},\tilde{W}}\cong \hat{\fg}_{\tilde{Q},\tilde{W}}\otimes\HO_{\BC^*}
\end{equation}
as well as an isomorphism of Lie algebras
\[
\hat{\fg}_{\tilde{Q},\tilde{W}}\cong \hat{\fg}^{\BC^*}_{\tilde{Q},\tilde{W}}/\fm_{\BC^*}\cdot \hat{\fg}^{\BC^*}_{\tilde{Q},\tilde{W}}.
\]
\end{proposition}
\begin{proof}
Before we start, we remind the reader that $\HO_{\BC^*}\neq \HO_{\BC^*_u}$, although these algebras are isomorphic: the first algebra has to do with the extra $\BC^*$-action on $\BC \tilde{Q}$, the second has to do with the action studied at the end of \S \ref{CoHAsec}.
\smallbreak
Since the coproduct is $\HO_{\BC^*}$-linear, and $\CoHA^{\BC^*,\chi}_{\tilde{Q},\tilde{W}}$ is free as a $\HO_{\BC^*}$-module by Proposition \ref{pur_form}, it follows that $\alpha$ is primitive if and only if $t\cdot \alpha$ is primitive, for $t$ a generator of $\fm_{\BC^*}$, the maximal homogeneous ideal in $\HO_{\BC^*}$.  Therefore, the space of primitives for the coproduct (i.e. $\hat{\fg}^{\BC^*}_{\tilde{Q},\tilde{W}}$) is free as a $\HO_{\BC^*}$-module.  The graded dimensions of this module are determined by the existence of the isomorphisms \eqref{dmm} and \eqref{diso}, yielding \eqref{dgiso}.  
\smallbreak
Again by $\HO_{\BC^*}$-linearity of the coproduct, the composition
\[
\hat{\fg}^{\BC^*}_{\tilde{Q},\tilde{W}}/\fm_{\BC^*}\cdotsh \hat{\fg}^{\BC^*}_{\tilde{Q},\tilde{W}}\hookrightarrow \CoHA^{\BC^*,\chi}_{\tilde{Q},\tilde{W}}/\fm_{\BC^*}\cdotsh\CoHA^{\BC^*,\chi}_{\tilde{Q},\tilde{W}}\xrightarrow{\cong}\CoHA^{\chi}_{\tilde{Q},\tilde{W}}
\] 
realises $\hat{\fg}^{\BC^*}_{\tilde{Q},\tilde{W}}/\fm_{\BC^*}\cdotsh\hat{\fg}^{\BC^*}_{\tilde{Q},\tilde{W}}$ as a subspace of primitive elements in $\CoHA^{\chi}_{\tilde{Q},\tilde{W}}$.  But the graded dimensions of $\hat{\fg}^{\BC^*}_{\tilde{Q},\tilde{W}}/\fm_{\BC^*}\cdotsh\hat{\fg}^{\BC^*}_{\tilde{Q},\tilde{W}}$ match those of $\hat{\fg}_{\tilde{Q},\tilde{W}}$, by the first statement, and the second statement follows.
\end{proof}
\subsection{Affinized finite type BPS algebras}
Let $Q$ be an orientation of a finite type, ADE Dynkin diagram.  We calculate the (undeformed) affine BPS Lie algebra $\hat{\fg}_{\tilde{Q},\tilde{W}}$.  We start with an easy lemma:
\begin{lemma}
\label{simpl}
Let $\CH$ be an algebra object in $\Heis\lmod_{\BZ}$, such that there is an isomorphism of underlying $\Heis_{\BZ}$-modules
\[
\CH\cong \bigoplus_{m\in 2\cdot\BZ_{\geq 1}}\Gr_{\bullet}^F(V_{[m]})[-2]^{\oplus a_m}
\]
for integers $a_m$.  Then up to isomorphism, $\CH$ is determined by $\CH^0$, the summand of $\CH$ of cohomological degree zero, considered as an ordinary associative algebra.  Similarly, if $\fg$ is a Lie algebra object in $\Heis\lmod_{\BZ}$ such that there is an isomorphism of underlying $\Heis_{\BZ}$-modules
\[
\fg\cong \bigoplus_{m\in 2\cdot \BZ_{\geq 1}}\Gr_{\bullet}^F(V_{[m]})[-2]^{\oplus b_m}
\]
for integers $b_m$, then $\fg$ is determined up to isomorphism by $\fg^0$, considered as an ordinary Lie algebra.
\end{lemma}
\begin{proof}
Under the conditions of the lemma, $\CH$ is concentrated in cohomological degrees indexed by $2\cdotsh\BN$.  Moreover, the lowering operator $q\colon \CH^{i}\rightarrow \CH^{i-2}$ is injective for $i>0$, since $\CH^{\geq 2}$ is spanned by elements $z^mD^n$, with $m,n\geq 1$, on which $q$ acts by sending $z^mD^n\mapsto nz^mD^{n-1}$.  As such, if $\alpha,\beta\in \CH$ are not both of cohomological degree zero, $\alpha\star \beta$ is determined by $q(\alpha\star\beta)=q(\alpha)\star\beta+\alpha\star q(\beta)$, and the result follows by induction.  The proof for $\fg$ is identical.
\end{proof}
Given a $\BN^{Q_0}$-graded Lie algebra $\fg$ we form the $\BN^{Q_0}$-graded Lie algebra $\fg\otimes_{\BQ} \BQ[D]$ by $\BQ[D]$-linear extension.  We make this into a $\Heis_Q$-module by setting $q_{\Sigma}(g\otimes D^m)=mg\otimes D^{m-1}$ and $p_i(g\otimes D^m)=\dd_ig\otimes D^{m+1}$, for $g$ of homogeneous $\BN^{Q_0}$-degree $\dd$.  The $\Heis_Q$-module $\fg\otimes_{\BQ} \BQ[D]$ inherits a $\Heis$-module structure via \eqref{Heis_forg}, and is $\BZ$-graded by (doubling the) $D$-degree.  For $Q$ a finite type ADE Dynkin diagram with associated Lie algebra $\fg$ and negative half $\fn_Q^-\subset \fg$, the Lie algebra object $\fn_Q^-\otimes_{\BQ}\BQ[D]$ may thus be considered as a Lie algebra object in $\Heis\lmod_{\BZ}$ satisfying the assumptions of Lemma \ref{simpl}.
\begin{proposition}
Let $Q$ be a quiver obtained by orienting a finite type Dynkin diagram.  There is an isomorphism of Lie algebras
\[
\hat{\fg}_{\tilde{Q},\tilde{W}}\cong \fn_{Q}^-\otimes_{\BQ}\BQ[D]
\]
where $\fn_Q^-$ is half of the simple Lie algebra associated to the underlying Dynkin diagram of $Q$.
\end{proposition}
\begin{proof}
It follows from \eqref{Kac_ADE} and \eqref{KactoDT} that $\fg^i_{\tilde{Q},\tilde{W}}=0$ for $i\neq 0$, and so $\hat{\fg}^i_{\tilde{Q},\tilde{W}}=0$ for $i<0$.  By Proposition \ref{dzprop} there is an isomorphism $\hat{\fg}_{\tilde{Q},\tilde{W}}^0\cong \fn_{Q}^-$, and then the result follows from Lemma \ref{simpl}.
\end{proof}
We leave it to the enthusiastic reader to calculate the deformed affine BPS Lie algebra $\fg_{\tilde{Q},\tilde{W}}^{\BC^*}$.
\subsection{Nakajima quiver varieties}
\label{NakSec}
Fix a finite quiver $Q$, and a framing vector $\ff\in \BN^{Q_0}$.  We define the quiver $Q_{\ff}$ to be the quiver obtained by adding a single vertex $\infty$ to the set $Q_0$, and for each $i\in Q_0$, adding $\ff_i$ arrows from $\infty$ to $i$, which we label $a_{i,m}$ for $m=1,\ldots ,\ff_i$.  We define $\widetilde{Q_{\ff}}$ to be the quiver obtained by tripling $Q_{\ff}$, and set $Q^+_{\ff}$ to be the quiver obtained by removing the loop $\omega_{\infty}$ from $\widetilde{Q_{\ff}}$.  We define 
\[
W^+=\left(\sum_{a\in Q_1}[a,a^*]+\sum_{i\in Q_0}\sum_{m=1}^{\ff_i}a_{i,m}a^*_{i,m}\right)\left(\sum_{i\in Q_0}\omega_i\right).
\]
Given a dimension vector $\dd\in\BN^{Q_0}$ we extend $\dd$ to a dimension vector $\dd^+$ for $Q^+_{\ff}$ by setting $\dd^+_{\infty}=1$.  We define $\BA_{Q^+_{\ff},\dd^+}^{\zeta\sst}\subset \BA_{Q^+_{\ff},\dd^+}$ to be the subspace of representations $\CF$ such that the one-dimensional vector space $e_{\infty}\cdot \CF$ generates $\CF$ under the action of $\BC Q^+_{\ff}$.  We call such modules $\zeta$-stable.  Then we define the scheme
\[
\Msp_{\ff,\dd}(Q)=\BA_{Q^+_{\ff},\dd^+}^{\zeta\sst}/\Gl_{\dd}.
\]
This is the moduli scheme parameterising pairs of a $\dd^+$-dimensional $\zeta$-stable $\BC Q^+_{\ff}$-module $\CF$, along with a trivialisation $e_{\infty}\cdotsh\CF\cong\BC$.  
\smallbreak
We identify $\BA_{\overline{Q_{\ff}},\dd^+}$ with the space of tuples $((A_a,A_a^*)_{a\in Q_1},(I_{i,m},J_{i,m})_{i\in Q_0, m\leq \ff_i})$ where $A_a\in \Hom(\BC^{\dd_{s(a)}},\BC^{\dd_{t(a)}})$, $A^*_a\in \Hom(\BC^{\dd_{t(a)}},\BC^{\dd_{s(a)}})$, $I_{i,m}\in (\BC^{\dd_i})^*$ and $J_{i,m}\in \BC^{\dd_i}$.  We define $\BA_{\overline{Q_{\ff}},\dd^+}^{\zeta\sst}\subset \BA_{\overline{Q_{\ff}},\dd^+}$ to be the subspace of representations $\CF$ of the doubled framed quiver such that the one-dimensional space $e_{\infty}\cdotsh\CF$ generates $\CF$ under the action of $\BC\overline{Q_{\ff}}$.  We consider the function 
\begin{align*}
\mu_{\dd}\colon &\BA_{\overline{Q_{\ff}},\dd^+}^{\zeta\sst}\rightarrow \mathfrak{gl}_{\dd}\\
&((A_a,A_a^*)_{a\in Q_1},(I_{i,m},J_{i,m})_{i\in Q_0, m\leq \ff_i})\mapsto \sum_{a\in Q_1}[A_a,A_a^*]+\sum_{i\in Q_0, m\leq \ff_i}I_{i,m}(J_{i,m}).
\end{align*}
The group $\Gl_{\dd}$ acts freely on $\BA_{\overline{Q_{\ff}},\dd^+}^{\zeta\sst}$, and thus on $\mu^{-1}_{\dd}(0)$, and we define the (smooth) \textit{Nakajima quiver variety} $\CN_{\ff,\dd}(Q)=\mu_{\dd}^{-1}(0)/\Gl_{\dd}$ \cite{Nak94}.  Composing the morphisms $\mu_{\dd}^{-1}(0)/\Gl_{\dd}\rightarrow \B \Gl_{\dd}\xrightarrow{\Det} \B \BC^*_u$ defines an action of $\HO_{\BC^*_u}$ on $\HO(\CN_{\ff,\dd}(Q),\BQ^{\vir})$ as in \eqref{udef}; by abuse of notation, we denote this action also by $\bullet$.  Via dimensional reduction, there is a natural isomorphism (see \cite[Prop.6.3]{preproj3})
\begin{equation}
\label{NDR}
\HO(\CN_{\ff,\dd}(Q),\BQ^{\vir})\cong \HO(\Msp_{\ff,\dd}(Q),\phip{\Tr(W^+)}).
\end{equation}
Let $\dd',\dd''\in\BN^{Q_0}$ be a pair of dimension vectors satisfying $\dd'+\dd''=\dd$.  We define $\Msp_{\ff,\dd',\dd''}(Q)$ to be the moduli scheme of pairs of a short exact sequence
\[
0\rightarrow \CF'\rightarrow \CF\rightarrow \CF''\rightarrow 0
\]
along with a trivialisation $e_{\infty}\cdotsh \CF\cong \BC$, where $\CF$ is a $\dd^+$-dimensional $\zeta$-stable $\BC Q^+_{\ff}$-module, and $\CF'$ is a $\dd'$-dimensional $\BC \tilde{Q}$-module, considered as a $\BC Q^+_{\ff}$-module with dimension vector zero at the vertex $\infty$.  The module $\CF''$ appearing in such a short exact sequence is a $\dd''^+$-dimensional $\zeta$-stable $\BC Q^+_{\ff}$-module.  We consider the standard correspondence diagram
\[
\xymatrix{
\Mst_{\dd'}(\tilde{Q})\times \Msp_{\ff,\dd''}(Q)&\ar[l]_-{\pi_1\times \pi_3}\Msp_{\ff,\dd',\dd''}(Q)\ar[r]^{\pi_2}&\Msp_{\ff,\dd}(Q).
}
\]
Composing the Thom--Sebastiani isomorphism with the pushforward $\pi_{2,\star}$ and pullback $(\pi_1\times\pi_3)^{\star}$ we define the operation
\[
\star\colon \HO\left(\Mst_{\dd'}(\tilde{Q}),\phip{\Tr(\tilde{W})}\right)\otimes \HO\left(\Msp_{\ff,\dd''}(Q),\phip{\Tr(W^+)}\right)\rightarrow \HO\left(\Msp_{\ff,\dd}(Q),\phip{\Tr(W^+)}\right).
\]
\begin{proposition}
\label{comp_Nak}
Let $Q$ be a finite quiver.
\begin{itemize}
\item
Via the isomorphism \eqref{NDR}, the operation $\star$ defines an action of the CoHA $\CoHA^{\chi}_{\tilde{Q},\tilde{W}}$ on the vector space 
\[
M_{\ff}(Q)\coloneqq \bigoplus_{\dd\in \BN^{Q_0}} \HO(\CN_{\ff,\dd}(Q),\BQ^{\vir})
\]
\item
The operation is $\HO_{\BC^*_u}$-linear: there are equalities $u\bullet (\alpha\star \beta)=(u\bullet \alpha)\star\beta+\alpha\star(u\bullet \beta)$.
\end{itemize}
\end{proposition}
\begin{proof}
For the first part, see e.g. \cite{So16,YZ18}.  The second part is proved as in Proposition \ref{Aproof} below.
\end{proof}
Now fix $Q=Q^{(1)}$ the one loop quiver, and fix the framing vector $\ff=1$.  Then we have $\CoHA_{\tilde{Q},\tilde{W},1}\cong \HO_{\BC^*}\cong \BQ[x]$.  One shows as in \cite{preproj3} that for $1_1$ the copy of $1\in\BQ[x]\cong \CoHA_{\tilde{Q},\tilde{W},1}$ we have
\begin{equation}
\label{basicfp}
1_1\star-=\fp_1
\end{equation}
where $\fp_1$ is the Nakajima raising operator \cite{Nak97} on $M_{1}(Q)\cong \bigoplus_{n\in\BN}\HO(\Hilb_n(\BA^2),\BQ^{\vir})$.  From the second part of Proposition \ref{comp_Nak} it follows that 
\begin{equation}
\label{basicder}
(u\bullet \alpha)\star(-)=[u\bullet,\alpha\star](-)
\end{equation}
as operators on $M_{1}(Q)$.  We recall the following theorem of Manfred Lehn:
\begin{theorem}[\cite{Lehn99} Thm.3.10]
\label{LehnThm}
Write $\fp_n$ for Nakajima's $n$th raising operator on $\Hilb(\BA^2)$.  Write $\fp'_n=[u\bullet,\fp_n]$.  Then
\[
[\fp'_1,\fp_n]=-n\fp_{n+1}.
\]
\end{theorem}
We refer to \cite{Nak97,Lehn99} for the definition of the operators $\fp_n$.  Alternatively, via \eqref{basicder}, \eqref{basicfp} and Theorem \ref{LehnThm}, these operators may be defined in terms of the action of $x^0\in\CoHA_{Q^{(3)},\tilde{W},1}$ and $u\bullet$ on $M_{1}(Q)$.

\subsection{Affine Yangian and shuffle algebra embedding}
\label{SHY_sec}
Let $T$ be a torus, acting on the quiver $Q$.  The superpotential $W=0$ is obviously $T$-invariant, and the vanishing cycle sheaf $\phip{\Tr(W)}\BQ_{\Mst^T_{\dd}(Q)}^{\vir}$ on $\Mst^T_{\dd}(Q)$ is isomorphic to the constant (shifted) perverse sheaf $\BQ_{\Mst^T_{\dd}(Q)}^{\vir}$.  It is thus straightforward to describe the vanishing cycle cohomology (leaving out the overall cohomological shifts for now):
\begin{align*}
\HO(\Mst^T_{\dd}(Q),\phip{\Tr(W)})\cong &\HO(\pt/(\Gl_{\dd}\times T),\BQ)\\
\cong &\BQ[x_{1,1},\ldots x_{1,\dd_1},x_{2,1},\ldots,x_{r,\dd_r}]^{\mathfrak{S}_{\dd}}\otimes \BQ[t_1,\ldots t_l]
\end{align*}
where $Q_0=\{1,\ldots ,r\}$ and $\mathfrak{S}_{\dd}=\prod_{i\in Q_0} \mathfrak{S}_{\dd_i}$ is the product of symmetric groups generated by permutations of the variables preserving their first subscript.  

\begin{proposition}\label{tfree}
Let the torus $T=\Hom_{\mathrm{Grp}}(\BZ,(\BC^*)^2)$ act on $\tilde{Q}$ via the function $\tilde{\tau}$ from \eqref{ttaudef}.  The natural morphism $\Phi\colon \CoHA^T_{\tilde{Q},\tilde{W}}\rightarrow \CoHA^T_{\tilde{Q}}$ of algebras is an embedding of algebras.
\end{proposition}
The definition of $\phi$ and proof that $\Phi$ preserves the algebra product is given in \cite{ScVa13}.  The fact that it is an embedding of algebras is proved in \cite[Thm.10.2]{preproj}, slightly modifying and fixing a gap in the proof in the published version of \cite[Prop.4.6]{ScVa20}.
\smallbreak
By \eqref{TZ} there is a (non-canonical) isomorphism of $\HO_{T'}$-modules
\begin{equation}
\label{savet}
\CoHA^{T'}_{\BA^2}\cong \CoHA_{\BA^2}\otimes \HO_{T'}.
\end{equation}
We denote by $\alpha^{(0)}_1\in\hat{\mathfrak{g}}^T_{\mathbb{A}^2,1}=\CoHA^{T'}_{\BA^2,1}$ the unique (up to scaling) non-zero element of cohomological degree $2$, and set $\alpha_1^{(m)}\coloneqq u^m\bullet \alpha_1^{(0)}$.
\smallbreak
For what follows, we introduce (the strictly positive part of) the affine Yangian $\CY_{t_1,t_2,t_3}(\widehat{\mathfrak{gl}(1)})$, for which our reference is \cite{Tsy17}.  Let $t_1,t_2,t_3$ be formal parameters satisfying $t_1+t_2+t_3=0$.  Then $\CY_{t_1,t_2,t_3}(\widehat{\mathfrak{gl}(1)})^+$ is the algebra over $\BQ[t_1,t_2,t_3]/(t_1+t_2+t_3)$ generated by symbols $e_i$ with $i\in \BN$  subject to the two relations 
\begin{align}
\label{id1}&[e_{i+3},e_j]-3[e_{i+2},e_{j+1}]+3[e_{i+1},e_{j+2}]-[e_i,e_{j+3}]\\
&+\sigma_2([e_{i+1},e_j]-[e_i,e_{j+1}])=-\sigma_3(e_ie_j+e_je_i) \nonumber\\
&\mathrm{Sym}_{\mathfrak{S}_3}[e_{i_1},[e_{i_2},e_{i_3+1}]]=0\label{id2}
\end{align}
where $\sigma_2=t_1t_2+t_1t_3+t_2t_3$ and $\sigma_3=t_1t_2t_3$.  
\begin{lemma}
\label{fsave}
As a $\HO_T$-module, $\CY_{t_1,t_2,t_3}(\widehat{\mathfrak{gl}(1)})$ is free.
\end{lemma}
\begin{proof}
We define a filtration $G$ on $\CY_{t_1,t_2,t_3}(\widehat{\mathfrak{gl}(1)})$ by putting $e_i$ in degree $-2+2i$.  So $G^r\CY_{t_1,t_2,t_3}(\widehat{\mathfrak{gl}(1)})$ is spanned by Lie words $[e_{i_1},[e_{i_2},[\ldots,[e_{i_{l-1}},e_{i_{l}}]\ldots]$ with $\sum_{j=1}^l(2i_j-2)\leq r$.  Set $\CA=\Gr^G\CY_{t_1,t_2,t_3}(\widehat{\mathfrak{gl}(1)})$.  Then $\CA=\UEA(\CL)\otimes \HO_T$ where $\CL$ is the Lie algebra over $\BQ$ generated by symbols $\overline{e}_i$ with $i\in \BN$ subject to the relations
\begin{align}
\label{id1a}&[\overline{e}_{i+3},\overline{e}_j]-3[\overline{e}_{i+2},\overline{e}_{j+1}]+3[\overline{e}_{i+1},\overline{e}_{j+2}]-[\overline{e}_i,\overline{e}_{j+3}]\\
&\mathrm{Sym}_{\mathfrak{S}_3}[\overline{e}_{i_1},[\overline{e}_{i_2},\overline{e}_{i_3+1}]]=0.\label{id2a}
\end{align}
In particular, $\CA$ is free as a $\HO_T$-module, from which it follows that $\CY_{t_1,t_2,t_3}(\widehat{\mathfrak{gl}(1)})$ is too.
\end{proof}

%The shuffle algebra embedding $\Phi$ is used in \cite{RSYZ20} to prove that the relations \eqref{id1} and \eqref{id2} hold in $\CoHA_{\BA^2}^{T}$:
\begin{proposition}\cite[Thm.7.1.1]{RSYZ20}
\label{RSYZ_thm}
There is an algebra morphism 
\[
\iota\colon\CY_{t_1,t_2,t_3}(\widehat{\mathfrak{gl}(1)})^+\rightarrow \CoHA_{\BA^2}^{T}
\]
sending $e_i$ to $\alpha_1^{(i)}$.  
\end{proposition}

Note that \cite[Thm.7.1.1]{RSYZ20} uses in an essential way the injectivity statement of Proposition \ref{tfree}, meaning that it is sufficient to check that the defining relations of $\CY_{t_1,t_2,t_3}(\widehat{\mathfrak{gl}(1)})^+$ hold after applying $\Phi$.  Since we work with the full rank 2-torus $T$ we may indeed apply \cite[Thm.10.2]{preproj} for this injectivity.  It is furthermore claimed in \cite{RSYZ20} that $\iota$ is injective, although the proof seems to be missing\footnote{Thanks go to the anonymous referee for spotting this gap in the literature.};  injectivity of $\iota$ follows from the following result of Tsymbaluk, as we will see in \S \ref{final_sec}.  
\begin{proposition}\cite[App.C]{Tsy17}
\label{Tsym_save}
The morphism
\[
\iota'\colon (\CY_{t_1,t_2,t_3}(\widehat{\mathfrak{gl}(1)})^+)_{t_2=0}\rightarrow \UEA_{\BQ[t]}(\Rees_F[W^+_{1+\infty}])
\]
sending $e_i$ to $z\tilde{D}^{(i)}$ is an injection.
\end{proposition}
We include here the explicit description of the shuffle algebra product on $\CoHA^T_{Q^{(3)}}$.  Since there is only one vertex, we can simplify the subscripts and write 
\[
\HO(\Mst^T_{d}(Q^{(3)}),\BQ)\cong \BQ[t_1,t_2]\otimes \BQ[x_1,\ldots,x_{d}]^{\mathfrak{S}_{d}}.
\]
We define $\zeta(z)\coloneqq(z+t_1)(z+t_2)(z-t_1-t_2)z^{-1}$.  Given two functions $f(x_1,\ldots,x_d)\in \HO(\Mst^T_{d}(Q^{(3)}),\BQ)$ and $g(x_1,\ldots,x_e)\in \HO(\Mst^T_{e}(Q^{(3)}),\BQ)$ we have
\begin{equation}
\label{shuff_def}
f\star g(x_1,\ldots,x_{d+e})=\sum_{\pi\in\mathrm{Sh}(d,e)}\pi\left(f(x_1,\ldots,x_d)g(x_{d+1},\ldots,g_{d+e})\prod_{\substack{1\leq i\leq d\\d+1\leq j\leq d+e}}\zeta(x_i-x_j)\right)
\end{equation}
where $\mathrm{Sh}(d,e)\subset \mathfrak{S}_{d+e}$ is the subset of the permutation group containing those $\pi$ such that $\pi(i)\leq \pi(i+1)$ for $i\neq d$.
\smallbreak
In \cite{Neg16} a related shuffle algebra is studied, in reference to the AGT relations\footnote{Many thanks to Andrei Negu\c{t} for patiently explaining the following to me.}.  Firstly, define $\zeta'(z)\coloneqq(z+t_1)(z+t_2)(z+t_1+t_2)^{-1}z^{-1}$.  Then we fix the coefficient field $\mathbb{F}=\mathbb{Q}(t_1,t_2)$ and for each $d$ define 
\[
\CoHA'^T_{\mathbb{A}^2,d,\mathrm{loc}}\coloneqq \mathbb{F}(x_1,\ldots x_d)^{\mathfrak{S}_d}
\]
the field of symmetric rational functions.  Then $\CoHA'^T_{\mathbb{A}^2,\mathrm{loc}}\coloneqq \bigoplus_{d\in\mathbb{Z}_{\geq 0}}\CoHA'^T_{\mathbb{A}^2,d,\mathrm{loc}}$ carries a shuffle algebra structure defined as in \eqref{shuff_def} but with $\zeta'(z)$ replacing $\zeta(z)$.  We define the algebra isomorphism
\begin{align*}
J\colon \CoHA^T_{\mathbb{A}^2,\mathrm{loc}}&\rightarrow \CoHA'^T_{\mathbb{A}^2,\mathrm{loc}}\\
f(x_1,\ldots,x_d)&\mapsto\prod_{1\leq i<j\leq d}(z^2-(t_1+t_2)^2)^{-1}f(x_1,\ldots,x_d)(t_1t_2)^{-\deg(f)}.
\end{align*}
Composing with the injections 
\[
\CoHA_{\mathbb{A}^2}^T\xhookrightarrow{i_1} \BQ[t_1,t_2]\otimes \BQ[x_1,\ldots,x_{d}]^{\mathfrak{S}_{d}}\xhookrightarrow{i_2} \CoHA^T_{\mathbb{A}^2,\mathrm{loc}}
\]
we obtain the injection $\Gamma=J\circ i_2\circ i_1$.  In \cite{Neg16}, explicit elements $\tilde{B}_d$ and $\tilde{L}_d$ of a Heisenberg--Virasoro algebra are given: it is easy to see that $\tilde{B}_1=\Gamma(1)$ and $\tilde{L}_1=\Gamma(x_1)$.  Then $\tilde{B}_d$ is equal to a generator of the 1-dimensional vector space $\Gamma(\mathfrak{g}_{\mathbb{A}^2,d})$ by Equation (2.43) of \cite{Neg16} and Proposition \ref{LehnWork} below.
\section{Heisenberg actions on $\CoHA_{Q,W}^{(\psi)}$}
\label{HeAsec}
It turns out that $\Heis_{Q}$-actions on cohomological Hall algebras exist at a very high level of generality\footnote{Note that the Heisenberg algebra appearing in this section has a different origin to the half Heisenberg algebra occurring inside $\hat{\fg}_{\mathbb{A}^2}$, studied in \S \ref{final_sec}.}.  We explain this first, before zooming in on the CoHAs that this paper is most concerned with, i.e. $\CoHA_{\BA^2}^T$ for various choices of $T$.
\subsection{Raising operators}
\begin{proposition}
\label{Aproof}
Let $Q$ be an arbitrary quiver, with a $T$-action for some torus $T=\Hom_{\mathrm{Grp}}(N,\BC^*)$, and with $T$-invariant potential $W\in\BC Q/[\BC Q,\BC Q]_{\mathrm{vect}}$.  For each $i\in Q_0$ the morphism $\alpha\mapsto u\bullet_i \alpha$, defined in \eqref{lidef} defines a derivation of $\CoHA^{T,{(\psi)}}_{Q,W}$.
\end{proposition}
\begin{proof}
Recall that the bracketed superscript $\psi$ in the statement of the proposition means that we claim the statement is true for both the algebras $\CoHA^{T}_{Q,W}$ and $\CoHA^{T,\psi}_{Q,W}$.  We first prove the version without the sign-twist, i.e. for the multiplication in $\CoHA^{T}_{Q,W}$.  

Fix $\dd',\dd''\in\BN^{Q_0}$ and set $\dd=\dd'+\dd''$.  Consider the following commutative diagram
\[
\xymatrix{
\ar[d]^{\Det_{(i)}\times\Det_{(i)}}\Mst^T_{\dd'}(Q)\times_{\B T}\Mst^T_{\dd''}(Q)&&\ar[ll]_-{\pi_1\times \pi_3}\Mst^T_{\dd',\dd''}(Q)\ar[rr]^{\pi_2}&&\Mst^T_{\dd}(Q)\ar[d]^{\Det_{(i)}}\\
\B \BC^*_u\times\B \BC^*_u\ar[rrrr]^-{\otimes}&&&&\B \BC^*_u.
}
\]
The morphisms $(\pi_1\times \pi_3)^{\star}$ and $\pi_{2,\star}$ respect the $\HO_{\BC^*_u}$-actions induced by the morphism to $\B \BC^*_u$ in the above diagram.  On the other hand, under the morphism
\[
\BQ[u]=\HO(\B \BC^*_u,\BQ)\xrightarrow{\otimes^{\star}} \BQ[u_1,u_2]=\HO(\B \BC^*_u\times\B \BC^*_u,\BQ)
\]
we have 
\begin{equation}
\label{DQ}
\otimes^{\star}\!(u)=u_1+ u_2.
\end{equation}
So 
\begin{align*}
u\bullet_i (m(\alpha\otimes \beta))=&m\left((\Det_{(i)}\otimes \Det_{(i)})^{\star}((u_1\otimes 1)+(1\otimes u_2))\cdot (\alpha\otimes \beta)\right)\\
=&m( ((u\bullet_i\alpha)\otimes \beta))+m((\alpha\otimes (u\bullet_i\beta))).
\end{align*}
If the multiplication is instead calculated in $\CoHA^{T,{\psi}}_{Q,W}$, all of the terms in the above expression are multiplied by $(-1)^{\psi(\dd',\dd'')}$, and so equality also holds in $\CoHA^{T,{\psi}}_{Q,W}$.
\end{proof}
\subsection{Lowering operators}
The morphism $\otimes^{\star}$ in the proof of Proposition \ref{Aproof} makes $\HO_{\BC^*_u}$ into a cocommutative coalgebra.  Let $\mathcal{L}$ be a line bundle on a scheme $X$, and let $\CF$ be a flat family of $\dd$-dimensional $\BC Q$-modules on $X$.  Then $\mathcal{L}\otimes \CF$ is a flat family of $\dd$-dimensional $\BC Q$-modules, where the $\BC Q$-action is induced by the given action on $\CF$ and the natural isomorphism $\mathcal{E}nd_X(\CL\otimes \CF,\CL\otimes\CF)\cong \CE nd_X(\CF,\CF)$.  Recall that an $X$-point of $\Mst_{\dd}^T(Q)$ is defined by a $T$-torsor $p\colon E\rightarrow X$ along with a flat family $\CF$ of $\dd$-dimensional $\BC Q$-modules on $E$ compatible with the $T$-action.  We define a morphism
\begin{align*}
A\colon &\B \BC^*_u\times \Mst^T_{\dd}(Q')\rightarrow \Mst^T_{\dd}(Q)\\
&(\CL, (E\xrightarrow{p} X,\CF))\mapsto (E\xrightarrow{p} X,p^*\CL\otimes \CF)
\end{align*}
which makes $\Mst^T_{\dd}(Q)$ into a module over the algebra object $\B \BC^*_u$ in stacks.  We obtain an induced morphism in vanishing cycle cohomology
\[
A^{\star}\colon \HO(\Mst^T_{\dd}(Q),\phip{\Tr(W)})\rightarrow\HO_{\BC^*_u}\otimes \HO(\Mst^T_{\dd}(Q),\phip{\Tr(W)})
\]
making $\CoHA^T_{Q,W,\dd}$ into a comodule for the coalgebra $\HO_{\BC^*_u}$.  Via our fixed identification $\HO_{\BC^*_u}=\BQ[u]$, for $\alpha\in \CoHA_{Q,W}$, we define coefficients $\alpha_n$ via the expansion
\begin{equation}
\label{Adef}
A^{\star}(\alpha)=\sum_{n\geq 0}u^n\otimes \alpha_n.
\end{equation}
\begin{proposition}
\label{lProp}
The morphism
\begin{align*}
\partial_u\colon &\CoHA^{T,(\psi)}_{Q,W}\rightarrow \CoHA^{T,(\psi)}_{Q,W}\\
&\alpha\mapsto \alpha_1
\end{align*}
defines a derivation of $\CoHA^{T,(\psi)}_{Q,W}$, where $\alpha_1$ is as defined in \eqref{Adef}.
\end{proposition}
\begin{proof}
We abbreviate $\Mst^T_{\dd}=\Mst^T_{\dd}(Q)$ etc. to reduce clutter.  As in the proof of Proposition \ref{Aproof}, it is sufficient to prove the statement for $\CoHA^{T}_{Q,W}$. Consider the commutative diagram
\[
\xymatrix{
\B\BC^*_u\times\Mst^T_{\dd'}\times_{\B T}\Mst^T_{\dd''}\ar[d]^a&\ar[l]\B \BC^*_u\times \Mst^T_{\dd',\dd''}\ar[r]\ar[d]^b&\B\BC^*_u\times \Mst^T_{\dd}\ar[d]^A\\
\Mst^T_{\dd'}\times_{\B T}\Mst^T_{\dd''}&\ar[l]\Mst^T_{\dd',\dd''}\ar[r]&\Mst^T_{\dd}
}
\]
with morphisms of points over a scheme $X$ defined by
\begin{align*}
a\colon &(\CL,(E\xrightarrow{p} X,\CF',\CF''))\mapsto (E\xrightarrow{p} X,p^*\CL\otimes \CF',p^*\CL\otimes \CF'')\\
b\colon &(\CL, (E\xrightarrow{p} X,\CF'\rightarrow \CF\rightarrow \CF''))\mapsto (E\xrightarrow{p} X, (p^*\CL\otimes \CF'\rightarrow p^*\CL\otimes \CF\rightarrow p^*\CL\otimes \CF'')).
\end{align*}
This induces the commutative diagram
\[
\xymatrix{
\HO_{\BC^*_u}\otimes \HO(\Mst^T_{\dd'}\times_{\B T} \Mst^T_{\dd''},\phip{\Tr(W)})\ar[rr]^-{\id\otimes m'_{\dd',\dd''}}&&\HO_{\BC^*_u}\otimes \HO(\Mst^T_{\dd},\phip{\Tr(W)})
\\
\ar[u]^{a^{\star}}\HO(\Mst^T_{\dd'}\times_{\B T}\Mst^T_{\dd''},\phip{\Tr(W)})\ar[rr]^-{m'_{\dd',\dd''}}&&\HO(\Mst^T_{\dd},\phip{\Tr(W)})\ar[u]^{A^{\star}}
}
\]
with $m'_{\dd',\dd''}$ as in \eqref{mdashdef}, from which we deduce that 
\[
(m'_{\dd',\dd''}(\alpha\otimes \beta))_1=m'_{\dd',\dd''}((\alpha\otimes \beta)_1)
\]
with $(\alpha\otimes \beta)_1$ defined via $a^*(\alpha\otimes \beta)=\sum_{n\geq 0} u^n\otimes (\alpha\otimes \beta)_n$.  Here we have also used that the pullback $\jmath^{\star}$ commutes with the $\HO_{\BC_u^*}$-action.  On the other hand $a$ factorises as 
\[
\B\BC_u^*\times\Mst^T_{\dd'}\times_{\B T}\Mst^T_{\dd''}\xrightarrow{\Delta\times \id\times\id}\B\BC^*_u\times \B\BC^*_u\times \Mst^T_{\dd'}\times_{\B T}\Mst^T_{\dd''}\xrightarrow{A\times A} \Mst^T_{\dd'}\times_{\B T}\Mst^T_{\dd''}
\]
and so since $\Delta^{\star}(p(u_1,u_2))=p(u,u)$ the $u$ coefficient of $a^*(\alpha\otimes \beta)$ is equal to $\alpha_1\otimes \beta+\alpha\otimes\beta_1$.
\end{proof}

\begin{proposition}
\label{HeisAct}
For each $\dd\in\BN^{Q_0}$ there is an action of $\Heis_{Q}$ on $\CoHA^T_{Q,W,\dd}$ defined by letting $p_i$ act by $u\bullet_i$, and letting $q_{\Sigma}$ act by $\partial_u$.  This action has central charge $2\dd$.
\end{proposition}
\begin{proof}
Consider the composition 
\begin{align*}
&\B \BC^*_u\times \Mst^T_{\dd}(Q)\xrightarrow{A} \Mst^T_{\dd}(Q)\xrightarrow{l_i}\B \BC^*_u\times  \Mst^T_{\dd}(Q),
\end{align*}
which on $X$-points is given by
\begin{align*}
&(\CL,(E\xrightarrow{p} X,\CF))\mapsto (\CL^{\otimes \dd_i}\otimes\Det(\CF_i),(E\xrightarrow{p} X,p^*\CL\otimes \CF)).
\end{align*}
The induced endomorphism $F=(l_i\circ A)^{\star}$ of $\HO_{\BC^*_u}\otimes \HO(\Mst_{\dd}(Q),\phip{\Tr(W)})$ satisfies
\begin{align*}
F(u\otimes \alpha)=(\dd_i(u\otimes 1)+(1\otimes u) )\bullet_i A^{\star}(\alpha)
\end{align*}
from which we deduce that
\begin{align*}
\partial_u (u\bullet_i\alpha)=&(u\bullet_i \alpha)_1\\
=&\dd_i \alpha_0+u\bullet_i \alpha_1\\
=&\dd_i\alpha+u\bullet_i(\partial_u \alpha)
\end{align*}
as required.
\end{proof}
Combining Propositions \ref{Aproof}, \ref{lProp} and \ref{HeisAct} yields
\begin{theorem}
\label{Heis_lift}
Let $Q$ be symmetric, so that $\CoHA^{T,(\psi)}_{Q,W}$ is a cohomologically graded algebra.  Then $\CoHA^{T,(\psi)}_{Q,W}$ is an algebra object in the category $\Heis_{Q}\lmod_{\BZ}$, for which the underlying $\Heis_{Q}$-module is integrable.  The decomposition
\[
\CoHA^{T,(\psi)}_{Q,W}\cong \bigoplus_{\dd\in\BN^{Q_0}}(\CoHA^{T,(\psi)}_{Q,W})_{[2 \dd]}
\]
according to central charge of the $\Heis_{Q}$-action is given by the standard decomposition \eqref{undobj} according to dimension vector.
\end{theorem}
\begin{proposition}
Let $Q$ be a symmetric quiver with potential $W\in \BC Q/[\BC Q,\BC Q]_{\mathrm{vect}}$.  The inclusion $\fg^T_{Q,W}\otimes \HO_{\BC^*_u}\hookrightarrow \CoHA^{T,\psi}_{Q,W}$ is an inclusion of $\BN^{Q_0}\oplus\BZ$-graded $\Heis_Q$-modules.
\end{proposition}
\begin{proof}
Since the image of $\fg^T_{Q,W}\otimes \HO_{\BC^*_u}$ is just the closure of the image of $\fg^T_{Q,W}\subset \CoHA^{T,\psi}_{Q,W}$ under the raising operator $u\bullet$, the proposition follows from the claim that the lowering operator $\partial_u$ annihilates $\fg^T_{Q,W}$.  Consider the following diagram
\[
\xymatrix{
\B \BC^*_u\times\Mst_{\dd}^T(Q)\ar[rd]_{p\coloneqq \JH^T\circ \pi_{2}}\ar[r]^-A&\Mst^T_{\dd}(Q)\ar[d]^{\JH^T}\\
&\Msp^T_{\dd}(Q)
}
\]
where $\pi_2$ is projection onto $\Mst_{\dd}^T(Q)$.  We claim that the diagram commutes: this is easy to see at the level of geometric points, from which the claim follows.  Thus there is a natural isomorphism 
\[
\JH^T_*A_*(\BQ\boxtimes \phip{\Tr(W)}\BQ^{\vir})\cong \HO(\B \BC^*_u,\BQ)\otimes \JH^T_*\phip{\Tr(W)}\BQ^{\vir}
\]
and so passing to the first perverse truncation and using the perverse vanishing result \eqref{Pvanishing} the adjunction morphism
\[
 {}^{\fp'}\!\!\tau^{\leq 1}\JH^T_*\phip{\Tr(W)}\BQ^{\vir}\rightarrow {}^{\fp'}\!\!\tau^{\leq 1}\JH^T_*A_*(\BQ\boxtimes \phip{\Tr(W)}\BQ^{\vir})
\]
is an isomorphism.  It follows that the morphism
\[
{}^{\fp'}\!\!\tau^{\leq 1}\JH^T_*\left(\phip{\Tr(W)}\BQ^{\vir}\rightarrow A_*(\BQ_{\B \BC^*_u}\boxtimes \phip{\Tr(W)}\BQ^{\vir})\right)
\]
factors through the morphism
\[
\HO^0(\B \BC^*_u,\BQ)\otimes {}^{\fp'}\!\tau^{\leq 1}\JH_*^T\phip{\Tr(W)}\BQ^{\vir}\rightarrow \HO(\B \BC^*_u,\BQ)\otimes {}^{\fp'}\!\tau^{\leq 1}\JH_*^T\phip{\Tr(W)}\BQ^{\vir}
\]
and $A^{\star}(\alpha)=1\otimes \alpha$ for $\alpha\in\FP_1\!\CoHA^{T,\psi}_{Q,W}$.  As such, $\alpha_1=0$, as required.
\end{proof}

\section{Proofs of main results}
\label{final_sec}
Let $Q=Q^{(1)}$ be the Jordan quiver, with one vertex and one loop, so $\tilde{Q}=Q^{(3)}$ is the quiver with three loops.  Then for every $q=p^n$ and every $d\geq 1$ there are $q$ absolutely indecomposable $d$-dimensional $\BF_q Q$-modules, so $\kac_{Q,d}(q)=q$, and
\begin{equation}
\label{BPSdim}
\fg_{Q^{(3)},\tilde{W},d}\cong \BQ[2]
\end{equation}
by \eqref{KactoDT}.  We define
\[
\fg_{\BA^2}^T\coloneqq \fg_{Q^{(3)},\tilde{W}}^T;\quad \quad\hat{\fg}_{\BA^2}^T\coloneqq \hat{\fg}_{Q^{(3)},\tilde{W}}^T
\]
the deformed and the deformed affinized BPS Lie algebras for the pair $(Q^{(3)},\tilde{W})$, respectively.  For $d\in\BN$, let $\alpha_{d}\in \fg^T_{\BA^2,d'}$ and $\alpha^{(0)}_{d}\in \hat{\fg}_{\BA^2}^T$ be the unique (up to scalar) nonzero elements of cohomological degree $-2$.
\begin{proposition}
Let the torus $T=\Hom_{\mathrm{Grp}}(\BZ,(\BC^*)^r)$ act on the quiver $Q^{(3)}$, with $\tilde{W}$ $T$-invariant.  Then the Lie bracket on $\fg^T_{\BA^2}$ vanishes.
\end{proposition}
\begin{proof}
By \eqref{BPSdim} and Theorem \ref{Tinj} there is an isomorphism of cohomologically graded vector spaces
\[
\fg^T_{\BA^2,d}\cong \HO_T[2]
\]
for every $d\geq 1$.  Then 
\[
[\alpha_{d'},\alpha_{d''}]\in \fg^T_{\BA^2,d'+d''}
\]
has cohomological degree $-4$, and so is zero.  For arbitrary $p(\mathbf{t}),q(\mathbf{t})\in\HO_T$ it follows that
\[
[p(\mathbf{t})\cdotsh\alpha_{d'},q(\mathbf{t})\cdotsh\alpha_{d''}]=p(\mathbf{t})q(\mathbf{t})\cdotsh[\alpha_{d'},\alpha_{d''}]=0.
\]
\end{proof}

By \eqref{BPSdim} we have the equality of generating series
\begin{align}
\nonumber\chi(\hat{\fg}_{\BA^2})\coloneqq &\sum_{i,j\in \BZ}\dim(\hat{\fg}_{\BA^2,i}^j)v^iq^j\\
=&q^{-2}v(1-v)^{-1}(1-q^2)^{-1}.\label{gen_serg}
\end{align}
From the isomorphism $\hat{\fg}_{\BA^2}^{\BC^*}\cong \hat{\fg}_{\BA^2}\otimes \HO_{\BC^*}$ (Proposition \ref{Tg}) we deduce the equality
\begin{equation}
\chi(\hat{\fg}^{\BC^*}_{\BA^2})=q^{-2}v(1-v)^{-1}(1-q^2)^{-2}.\label{gen_serg2}
\end{equation}

We continue to denote by $\alpha_{d}^{(0)}$ a basis element of $\hat{\fg}_{\BA^2,d}^{-2}$, and set $\alpha_{d}^{(m)}\coloneqq u^m\bullet \alpha_d^{(0)}$.  Then the elements $\alpha_d^{(m)}$ with $d\geq 1$ and $m\geq 0$ form a basis for $\hat{\fg}_{\BA^2}$.
\begin{proposition}
\label{LehnWork}
For every $d\geq 1$ we have an inequality
\begin{equation}
\label{ineqr}
[\alpha_1^{(1)},\alpha_{d}^{(0)}]= \lambda_d \alpha_{d+1}^{(0)}
\end{equation}
for some nonzero $\lambda_{d}\in\BQ$.
\end{proposition}
\begin{proof}
We consider the $\CoHA_{\BA^2}$-action on $M_{1}(Q)\cong\bigoplus_{d\in\BN}\HO(\Hilb_d(\BC^2),\BQ^{\vir})$ recalled in \S \ref{NakSec}.  We denote by $\rho\colon\hat{\fg}_{\BA^2}\rightarrow \mathrm{End}(M_1(Q))$ the action restricted to $\hat{\fg}_{\BA^2}$.  Then $\rho(\alpha_0^{(0)})=\fp_1$ is the action of the Nakajima raising operator and from 
\[
u\bullet (\alpha \star m)=(u\bullet \alpha)\star m+\alpha\star (u\bullet m)
\]
we deduce that 
\[
\rho(\alpha_1^{(1)}) =[u \bullet, \rho(\alpha_1^{(0)})]= \fp_1'.
\]
From Theorem \ref{LehnThm} we deduce inductively that 
\begin{equation}
\label{alphapin}
\rho\left(\mathbf{ad}_{\alpha_1^{(1)}}^n(\alpha_1^{(0)})\right)\neq 0.
\end{equation}
Set $\alpha=\mathbf{ad}_{\alpha_1^{(1)}}^d(\alpha_1^{(0)})\in \widehat{\fg}_{\BA^2,d+1}$.  By \eqref{alphapin}, $\alpha\neq 0$.  On the other hand, $\alpha_{d+1}^{(0)}$ is the unique nonzero element of $\widehat{\fg}_{\BA^2,d+1}$ of cohomological degree $-2$, which is the cohomological degree of $\alpha$, and the result follows by induction.
\end{proof}

\begin{corollary}
\label{cor1}
The affinized BPS Lie algebra $\hat{\fg}_{\BA^2}$ is generated by $\alpha_1^{(m)}$ for $m\geq 0$.
\end{corollary}
\begin{proof}
We assume, for an inductive argument, that the Lie subalgebra $\fg'\subset \hat{\fg}_{\BA^2}$ generated by $\alpha_1^{(m)}$ for $m\geq 0$ contains all of the elements $\alpha_{d'}^{(m)}$ for $d'<d+1$ and $m\in \BN$.  From Proposition \ref{LehnWork} we deduce that, possibly after replacing $\alpha_{d+1}^{(0)}$ by itself multiplied by some $\lambda\in \BQ\setminus\{0\}$, we have the equality $[\alpha_1^{(1)},\alpha_{d}^{(0)}]=\alpha_{d+1}^{(0)}$.  Acting by $u^m\bullet$, we find
\[
\sum_{0\leq r\leq m}\binom{m}{r}[\alpha_1^{(1+r)},\alpha_{d}^{(m-r)}]=\alpha_{d+1}^{(m)}.
\]
\end{proof}
The point of the next theorem and corollary is that the structure of $\CoHA_{\BA^2}$ follows from dimension counting, spherical generation, its upgrade to an algebra object in $\Heis\lmod_{\BZ}$ (\S \ref{HeAsec}) and from the construction of the affinized BPS Lie algebra (\S \ref{CopSec}).
\begin{theorem}
\label{fder}
There is an isomorphism $F\colon \hat{\fg}_{\BA^2}\cong \Gr^F_{\bullet}\!W^+_{1+\infty}$.
\end{theorem}
\begin{proof}
This is a corollary of \eqref{gen_serg}, Theorem \ref{Heis_lift}, Corollary \ref{cor1}, which states that $\hat{\fg}_{\BA^2}$ is generated by $\hat{\fg}_{\BA^2,1}=\Gr^F_{\bullet}\! V_{[2]}$, and Corollary \ref{gen_cor}.
\end{proof}
The next corollary follows from Proposition \ref{UEA_prop}.  The second statement in the corollary follows from Lemma \ref{SphGen}.
\begin{corollary}
\label{RH}
There is an isomorphism of algebras
\[
\CoHA_{\BA^2}\cong \UEA(\Gr^F_{\bullet}\!W^+_{1+\infty}).
\]
In particular, $\CoHA_{\BA^2}$ is spherically generated (i.e. it is generated by $\CoHA_{\BA^2,1}$).
\end{corollary}

\smallbreak
Consider the torus $T=(\BC^*)^2$, acting by independently rescaling the two coordinates of $\BA^2$ (via the weighting $\tilde{\tau}$ of \eqref{ttaudef}), and the fully equivariant CoHA $\CoHA^T_{\BA^2}$.%  For arbitrary $T'$, we abuse notation by continuing to denote by $\alpha_d\in \hat{\fg}_{\mathbb{A}^2,d}^{T'}$ the unique (up to scalar) non-zero element of cohomological degree $-2$.
\begin{theorem}
\label{TSG}
If $T'$ is a torus acting on $Q^{(3)}$, leaving $\tilde{W}$ invariant, then $\CoHA^{T'}_{Q^{(3)},\tilde{W}}$ is spherically generated.  In particular, the algebra $\CoHA^T_{\BA^2}$ is spherically generated.
\end{theorem}
\begin{proof}
For $d\geq 1$ we define a morphism of $\HO_{T'}$-modules
\begin{align*}
&\Psi_d\colon \BQ[u]\otimes \HO_{T'}\rightarrow \CoHA^{T'}_{\BA^2}\\
u^m\otimes t^l\mapsto&\begin{cases}
t^l\left([\alpha_1^{(1)},-]^{d-1}(\alpha_1^{(m)})\right)& \textrm{if }d-1<m\textrm{ or }m=0\\
t^l\left(([\alpha_1^{(1)},-]^{d-1-m}\circ [\alpha_1^{(0)},-]\circ [\alpha_1^{(1)},-]^{m-1})(\alpha_1^{(m+1)})\right)&\textrm{otherwise}
%(\mathbf{ad}_{zD})^{m-n}\circ \mathbf{ad}_{D}\circ (\mathbf{ad}_{zD})^{n-1}(zD^{n+1}) & \textrm{otherwise}
\end{cases}
\end{align*}
and define
\[
\Psi=\oplus_{d\geq 1} \Psi_d\colon \bigoplus_{d\geq  1} \BQ[u]\otimes \HO_{T'}\rightarrow \CoHA^{T'}_{\BA^2}.
\]
Finally we define
\begin{align*}
\Phi\colon \Sym_{\HO_{T'}}\!\left(\bigoplus_{d\geq 1} \BQ[u]\otimes \HO_{T'}\right)\rightarrow \CoHA^{T'}_{\BA^2}
\end{align*}
via $\Psi$ and the CoHA product on the target.  This is a morphism of free $\HO_{T'}$-modules that (via Lemma \ref{Blawan}, Corollary \ref{RH} and Proposition \ref{pur_form} \eqref{algeiso}) becomes an isomorphism after applying $\otimes_{\HO_{T'}}(\HO_{T'}/\fm_{T'})$, where $\fm_{T'}\subset \HO_{T'}$ is the maximal homogeneous ideal.  It follows that $\Phi$ is an isomorphism, and the algebra $\CoHA^{T'}_{\BA^2}$ is spherically generated.
\end{proof}

There is a partial converse to Theorem \ref{TSG}, provided by the next two propositions:
\begin{proposition}
Let $Q$ be a quiver without loops that is not an orientation of a finite type ADE Dynkin diagram, and let $T$ act on $\tilde{Q}$, leaving $\tilde{W}$ invariant.  Then $\CoHA^T_{\tilde{Q},\tilde{W}}$ is not spherically generated, i.e. it is not generated by the subspaces $\CoHA^T_{\tilde{Q},\tilde{W},1_i}$ for $i\in Q_0$.
\end{proposition}
\begin{proof}
This follows for cohomological degree reasons.  The conditions on $Q$ imply that $\kac_{Q,\dd}(q)$ is not a constant for some $\dd\in\BN^{Q_0}$.  From \eqref{TZ} and \eqref{KactoDT} we deduce that $\CoHA^T_{\tilde{Q},\tilde{W}}$ has a summand in strictly negative cohomological degree.  For each $i\in Q_0$, since there are no loops supported at $i$, there is an isomorphism
\[
\CoHA^T_{\tilde{Q},\tilde{W},1_i}\cong \HO_T
\]
as a cohomologically graded vector space.  So the sum of these spaces for $i\in Q_0$ generates an algebra lying in positive cohomological degrees.
\end{proof}
The next proposition is in the same vein.  Before we state it, we introduce some notation.  We label the three loops of $Q^{(3)}$ by the symbols $x,y,z$.  We let $\Mst^{\mathcal{SN}}(Q^{(3)})\subset\Mst(Q^{(3)})$ be the reduced closed substack, the points of which correspond to representations for which $z$ acts via a nilpotent operator, and let $\Mst^{\mathcal{N}}(Q^{(3)})\subset\Mst(Q^{(3)})$ be the reduced closed substack, the points of which correspond to representations for which both $y$ and $z$ act via nilpotent operators.  Then
\begin{align*}
\CoHA^{\heartsuit}_{Q^{(3)},\tilde{W}}\coloneqq \bigoplus_d\HO\!\left(\Mst_d^{\heartsuit}(Q^{(3)}),(\phip{\Tr(\tilde{W})}\BQ_{\Mst_d(Q^{(3)})}^{\vir})\lvert_{\Mst^{\heartsuit}(Q^{(3)})}\right)
\end{align*}
carry Hall algebra structures defined as in \S \ref{CoHAsec}, for $\heartsuit=\mathcal{SN},\mathcal{N}$.  If the torus $T'\cong (\mathbb{C}^*)^n$ acts on $\mathbb{A}^3$, preserving the 3-form $dx\wedge dy\wedge dz$ then there is again an induced action on $\CoHA^{T,\heartsuit}_{\mathbb{A}^2}\coloneqq \CoHA^{T,\heartsuit}_{Q^{(3)},\tilde{W}}$.  By the PBW theorem \cite[Thm.C]{QEAs}, there exist Lie subalgebras $\fg^{T,\heartsuit}_{\mathbb{A}^2}\subset \CoHA^{T,\heartsuit}_{\mathbb{A}^2}$ such that the induced morphism
\begin{equation}
\label{SNPBW}
\Sym_{\HO_T}\left(\fg^{T,\heartsuit}_{\mathbb{A}^2}\otimes\HO_{\BC^*_u}\right)\rightarrow \CoHA^{T,\heartsuit}_{\mathbb{A}^2}
\end{equation}
is an isomorphism.  The generating functions of these Lie algebras are given by
\begin{equation}
\label{SNgen}
\chi(\fg^{T,\heartsuit}_{\BA^2})=vq^{f(\heartsuit)}(1-v)^{-1}(1-q^2)^{-n}
\end{equation}
where $f(\mathcal{SN})=0$ and $f(\mathcal{N})=2$; see \cite{BSV17,preproj}.
\begin{proposition}
The algebras $\CoHA^{T,\heartsuit}_{\mathbb{A}^2}$ for $\heartsuit=\mathcal{SN},\mathcal{N}$ are not spherically generated.
\end{proposition}
\begin{proof}
Both results follow by dimension counting.  Firstly, both algebras are concentrated entirely in non-negative cohomological degrees.  Then \eqref{SNPBW} and \eqref{SNgen} imply that the subspace $V\subset \CoHA^{T,\mathcal{SN}}_{\mathbb{A}^2,2}$ in cohomological degree zero is 2 dimensional, while the cohomological degree zero piece of $U\subset \CoHA^{T,\mathcal{SN}}_{\mathbb{A}^2,1}$  is 1-dimensional, so cannot generate $V$.  The argument for $\CoHA^{T,\mathcal{N}}_{\mathbb{A}^2}$ is similar.
\end{proof}
It would be interesting to have some neat description of the images of any one of the injections $\CoHA_{\mathbb{A}^2}^{T,\heartsuit}\hookrightarrow \bigoplus_{d\geq 1}\BQ[t_1,t_2]\otimes \BQ[x_1,\ldots,x_{d}]^{\mathfrak{S}_{d}}$, for $\heartsuit=\emptyset,\mathcal{SN},\mathcal{N}$ to compare with the main result of \cite{Neg22} in K-theory for the case $\heartsuit=\mathcal{SN}$ --- we leave this to future work.  
\smallbreak
\begin{theorem}
\label{HW_cor}
Let $\BC^*$ act with weights $(1,0)$ on the two coordinates of $\BA^2$.  Then there is a $\HO_{\mathbb{C}^*}$-linear  isomorphism 
\[
F\colon\hat{\fg}^{\BC^*}_{\BA^2}\cong\Rees_F[W^+_{1+\infty}]
\]
between the deformed affinized BPS Lie algebra for the pair $(Q^{(3)},\tilde{W})$ and the Rees Lie algebra of $W^+_{1+\infty}$, uniquely determined by setting $F(\alpha_1^{(m)})=z\tilde{D}^{(m)}$.  Thus, there is an isomorphism of algebras 
\begin{equation}
\label{reqiso}
\Coha^{\BC^*}_{\BA^2}\cong \UEA_{\BQ[t]}(\Rees_F[W^+_{1+\infty}]).
\end{equation}
\end{theorem}
\begin{proof}
Let $T=\Hom_{\mathrm{Grp}}(\BZ,(\BC^*)^2)$ act on $\tilde{Q}$ via the weighting $\tilde{\tau}\colon \tilde{Q}_1\rightarrow \BZ^2$ from \eqref{ttaudef}.  By Proposition \ref{pur_form} there is an isomorphism of algebras
\begin{equation}
\label{hwgt}
\CoHA^{\BC^*}_{\BA^2}\cong \CoHA^{T}_{\BA^2}\otimes_{\HO_T}\BQ[t_1,t_2]/(t_2).
\end{equation}
Then by Proposition \ref{RSYZ_thm} there is a morphism of algebras 
\begin{align}
G\colon (\CY_{t_1,t_2,t_3}(\widehat{\mathfrak{gl}(1)})^+)_{t_2=0}\rightarrow \CoHA^{\BC^*}_{\BA^2}
\end{align}
setting $t_1=t$ and sending $e_i$ to $\alpha_{1}^{(i)}$.  This morphism is \textit{surjective} by Theorem \ref{TSG}.  By Proposition \ref{Tsym_save} there is an \textit{injection} $\iota'\colon (\CY_{t_1,t_2,t_3}(\widehat{\mathfrak{gl}(1)})^+)_{t_2=0}\hookrightarrow \UEA_{\BQ[t]}(\Rees_F[W^+_{1+\infty}])$.  Now by \eqref{gen_serg2} we have that
\[
\chi(\hat{\fg}^{\BC^*}_{\BA^2})=q^{-2}v(1-v)^{-1}(1-q^2)^{-2}=\chi(\Rees_F[W^+_{1+\infty}])_{v^2\mapsto v}
\] 
with the doubling of exponents as in Theorem \ref{Heis_lift}.  By Proposition \ref{TPBW} we deduce
\[
\chi(\CoHA^{\BC^*}_{\BA^2})=\chi(\UEA_{\BQ[t]}(\hat{\fg}^{\BC^*}_{\BA^2}))=\chi(\UEA_{\BQ[t]}(\Rees_F[W^+_{1+\infty}]))_{v^2\mapsto v}
\]
and both $G$ and $\iota'$ are isomorphisms.  The isomorphism \eqref{reqiso} is provided by $\iota'\circ G^{-1}$.  Composing $\iota'\circ G^{-1}$ with the PBW isomorphism (Proposition \ref{TPBW} again) yields the isomorphism
\[
\UEA_{\BQ[t]}(\hat{\fg}^{\BC^*}_{\BA^2}) \xrightarrow{\cong}\UEA_{\BQ[t]}(\Rees_F[W^+_{1+\infty}])
\]
restricting to the claimed isomorphism $F$.
\end{proof}

\begin{proposition}
\label{save_inj}
The morphism $\iota$ from Proposition \ref{RSYZ_thm} is injective, so that $\CY_{t_1,t_2,t_3}(\widehat{\mathfrak{gl}(1)})^+$ may be identified with the spherical subalgebra $\CS\!\CoHA^T_{\BA^2}\subset \CoHA^T_{\BA^2}$ under $\iota$, i.e. the subalgebra generated by $\CoHA^T_{\BA^2,1}$.
\end{proposition}
\begin{proof}
By construction, $\iota$ induces a surjection $\CY_{t_1,t_2,t_3}(\widehat{\mathfrak{gl}(1)})^+\rightarrow \CS\!\CoHA^T_{\BA^2}$.  The morphism $\iota$ is injective, since it is a morphism of free $\HO_T$-modules (by Proposition \ref{pur_form} and Lemma \ref{fsave}) and it is injective after setting $t_2=0$ by Proposition \ref{Tsym_save} and Theorem \ref{HW_cor}.
\end{proof}

Finally we can completely describe the fully equivariant CoHA:
\begin{theorem}
\label{Full_Yang}
There is an isomorphism of algebras $\iota\colon \CY_{t_1,t_2,t_3}(\widehat{\mathfrak{gl}(1)})^+ \rightarrow \CoHA^T_{\BA^2}$  sending $e_i$ to $\alpha_1^{(i)}$.
\end{theorem}
\begin{proof}
This is an immediate corollary of Proposition \ref{save_inj} and Theorem \ref{TSG}.
\end{proof}
%We can now give a more complete statement regarding the relation between BPS Lie algebras and $W^+_{1+\infty}$.

\bibliographystyle{alpha}
\bibliography{Literatur}

\begin{thebibliography}{BBBBJ15}

\bibitem[AS13]{Arb12}
N.~Arbesfeld and O.~Schiffmann.
\newblock A presentation of the deformed {$W_{1+\infty}$} algebra.
\newblock In {\em Symmetries, integrable systems and representations},
  volume~40 of {\em Springer Proc. Math. Stat.}, pages 1--13. Springer,
  Heidelberg, 2013.

\bibitem[BBBBJ15]{BBBD}
O.~Ben-Bassat, C.~Brav, V.~Bussi, and D.~Joyce.
\newblock A `{D}arboux theorem' for shifted symplectic structures on derived
  {A}rtin stacks, with applications.
\newblock {\em Geom. Topol.}, 19(3):1287--1359, 2015.

\bibitem[BBDG18]{BBD}
A.~Beilinson, J.~Bernstein, P.~Deligne, and Ofer Gabber.
\newblock Faisceaux pervers.
\newblock {\em Ast\'erisque}, (100):vi+180, 2018.

\bibitem[BBJ19]{BBD19}
C.~Brav, V.~Bussi, and D.~Joyce.
\newblock A {D}arboux theorem for derived schemes with shifted symplectic
  structure.
\newblock {\em J. Amer. Math. Soc.}, 32(2):399--443, 2019.

\bibitem[BBS13]{BBS}
K.~Behrend, J.~Bryan, and B.~Szendr\H{o}i.
\newblock Motivic degree zero {D}onaldson-{T}homas invariants.
\newblock {\em Invent. Math.}, 192(1):111--160, 2013.

\bibitem[Beh09]{Behrend}
K.~Behrend.
\newblock Donaldson-{T}homas type invariants via microlocal geometry.
\newblock {\em Ann. of Math. (2)}, 170(3):1307--1338, 2009.

\bibitem[BSV20]{BSV17}
T.~Bozec, O.~Schiffmann, and E.~Vasserot.
\newblock On the number of points of nilpotent quiver varieties over finite
  fields.
\newblock {\em Ann. Sci. \'Ec. Norm. Sup\'er. (4)}, 53(6):1501--1544, 2020.

\bibitem[Dav17]{Da13}
B.~Davison.
\newblock The critical {C}o{HA} of a quiver with potential.
\newblock {\em Q. J. Math.}, 68(2):635--703, 2017.

\bibitem[Dav23a]{Da22BF}
B.~Davison.
\newblock A boson-fermion correspondence in cohomological {D}onaldson-{T}homas
  theory.
\newblock {\em Glasg. Math. J.}, 65(S1):S28--S52, 2023.

\bibitem[Dav23b]{preproj}
B.~Davison.
\newblock The integrality conjecture and the cohomology of preprojective
  stacks.
\newblock {\em J. Reine Angew. Math.}, 804:105--154, 2023.

\bibitem[Dav25]{preproj3}
B.~Davison.
\newblock B{PS} {L}ie algebras and the less perverse filtration on the
  preprojective {C}o{HA}.
\newblock {\em Adv. Math.}, 463:Paper No. 110114, 75, 2025.

\bibitem[DM20]{QEAs}
B.~Davison and S.~Meinhardt.
\newblock Cohomological {D}onaldson-{T}homas theory of a quiver with potential
  and quantum enveloping algebras.
\newblock {\em Invent. Math.}, 221(3):777--871, 2020.

\bibitem[Gin06]{ginz}
V.~Ginzburg.
\newblock {Calabi--Yau algebras}.
\newblock {\url{https://arxiv.org/abs/math/0612139}}, 2006.

\bibitem[Gro96]{Groj96}
I.~Grojnowski.
\newblock Instantons and affine algebras. {I}. {T}he {H}ilbert scheme and
  vertex operators.
\newblock {\em Math. Res. Lett.}, 3(2):275--291, 1996.

\bibitem[GY22]{GaYa20}
D.~Galakhov and M.~Yamazaki.
\newblock Quiver {Y}angian and supersymmetric quantum mechanics.
\newblock {\em Comm. Math. Phys.}, 396(2):713--785, 2022.

\bibitem[Kac83]{Kac83}
V.~Kac.
\newblock Root systems, representations of quivers and invariant theory.
\newblock In {\em Invariant theory ({M}ontecatini, 1982)}, volume 996 of {\em
  Lecture Notes in Math.}, pages 74--108. Springer, Berlin, 1983.

\bibitem[KS90]{KSsheaves}
M.~Kashiwara and P.~Schapira.
\newblock {\em Sheaves on manifolds}, volume 292 of {\em Grundlehren der
  mathematischen Wissenschaften [Fundamental Principles of Mathematical
  Sciences]}.
\newblock Springer-Verlag, Berlin, 1990.
\newblock With a chapter in French by Christian Houzel.

\bibitem[KS08]{KS1}
M.~Kontsevich and Y.~Soibelman.
\newblock {Stability structures, motivic Donaldson--Thomas invariants and
  cluster transformations}.
\newblock 2008.
\newblock math.AG/08112435.

\bibitem[KS11]{KS2}
M.~Kontsevich and Y.~Soibelman.
\newblock Cohomological {H}all algebra, exponential {H}odge structures and
  motivic {D}onaldson-{T}homas invariants.
\newblock {\em Commun. Number Theory Phys.}, 5(2):231--352, 2011.

\bibitem[KV23]{KV19}
M.~Kapranov and E.~Vasserot.
\newblock The cohomological {H}all algebra of a surface and factorization
  cohomology.
\newblock {\em J. Eur. Math. Soc. (JEMS)}, 25(11):4221--4289, 2023.

\bibitem[Leh99]{Lehn99}
M.~Lehn.
\newblock Chern classes of tautological sheaves on {H}ilbert schemes of points
  on surfaces.
\newblock {\em Invent. Math.}, 136(1):157--207, 1999.

\bibitem[LY20]{LY20}
Wei Li and Masahito Yamazaki.
\newblock Quiver {Y}angian from crystal melting.
\newblock {\em J. High Energy Phys.}, (11):035, 124, 2020.

\bibitem[Mas01]{Ma01}
D.~Massey.
\newblock The {S}ebastiani-{T}hom isomorphism in the derived category.
\newblock {\em Compositio Math.}, 125(3):353--362, 2001.

\bibitem[Mik07]{Mi07}
K.~Miki.
\newblock A {$(q,\gamma)$} analog of the {$W_{1+\infty}$} algebra.
\newblock {\em J. Math. Phys.}, 48(12):123520, 35, 2007.

\bibitem[Nak94]{Nak94}
H.~Nakajima.
\newblock {Instantons on ALE spaces, quiver varieties, and Kac-Moody algebras}.
\newblock {\em Duke Math. J.}, 76(2):365--416, 1994.

\bibitem[Nak97]{Nak97}
H.~Nakajima.
\newblock Heisenberg algebra and {H}ilbert schemes of points on projective
  surfaces.
\newblock {\em Ann. of Math. (2)}, 145(2):379--388, 1997.

\bibitem[Neg16]{Neg16}
A.~Negu\c{t}.
\newblock Exts and the {AGT} relations.
\newblock {\em Lett. Math. Phys.}, 106(9):1265--1316, 2016.

\bibitem[Neg24]{Neg22}
A.~Negu\c{t}.
\newblock An integral form of quantum toroidal {$\mathfrak{gl}_1$}.
\newblock {\em Math. Rep. (Bucur.)}, 26(76)(3-4):183--205, 2024.

\bibitem[RS17]{RS17}
J.~Ren and Y.~Soibelman.
\newblock {Cohomological Hall Algebras, Semicanonical Bases and
  Donaldson--Thomas Invariants for 2-dimensional Calabi--Yau Categories (with
  an Appendix by Ben Davison)}.
\newblock In {\em Algebra, Geometry, and Physics in the 21st Century}, pages
  261--293. Springer, 2017.

\bibitem[RSYZ20]{RSYZ20}
M.~Rap{\v{c}}{\'a}k, Y.~Soibelman, Y.~Yang, and G.~Zhao.
\newblock {Cohomological Hall algebras, vertex algebras and instantons}.
\newblock {\em Comm. Math. Phys.}, 376(3):1803--1873, 2020.

\bibitem[RSYZ23]{RSYZ20b}
M.~Rap\v{c}\'ak, Y.~Soibelman, Y.~Yang, and G.~Zhao.
\newblock Cohomological {H}all algebras and perverse coherent sheaves on toric
  {C}alabi-{Y}au 3-folds.
\newblock {\em Commun. Number Theory Phys.}, 17(4):847--939, 2023.

\bibitem[Soi16]{So16}
Y.~Soibelman.
\newblock {Remarks on Cohomological Hall algebras and their representations}.
\newblock In {\em Arbeitstagung Bonn 2013}, pages 355--385. Springer, 2016.

\bibitem[SV13]{ScVa13}
O.~Schiffmann and E.~Vasserot.
\newblock {Cherednik algebras, W-algebras and the equivariant cohomology of the
  moduli space of instantons on $\mathbb{A}^2$}.
\newblock {\em Pub. Math. IH{\'E}S}, 118(1):213--342, 2013.

\bibitem[SV20]{ScVa20}
O.~Schiffmann and E.~Vasserot.
\newblock On cohomological {H}all algebras of quivers: generators.
\newblock {\em J. Reine Angew. Math.}, 760:59--132, 2020.

\bibitem[Sze08]{Sz08}
B.~Szendr\H{o}i.
\newblock {Non-commutative Donaldson-Thomas theory and the conifold}.
\newblock {\em Geom. Topol}, 12(2):1171--1202, 2008.

\bibitem[Tho00]{Thomas1}
R.~P. Thomas.
\newblock A holomorphic {C}asson invariant for {C}alabi-{Y}au 3-folds, and
  bundles on {$K3$} fibrations.
\newblock {\em J. Differential Geom.}, 54(2):367--438, 2000.

\bibitem[Tsy17]{Tsy17}
A.~Tsymbaliuk.
\newblock {The affine Yangian of $\mathfrak{gl}_1$ revisited}.
\newblock {\em Adv. Math.}, 304:583--645, 2017.

\bibitem[VdB15]{MR3338683}
M.~Van~den Bergh.
\newblock Calabi-{Y}au algebras and superpotentials.
\newblock {\em Selecta Math. (N.S.)}, 21(2):555--603, 2015.

\bibitem[YZ18]{YZ18}
Y.~Yang and G.~Zhao.
\newblock {The cohomological Hall algebra of a preprojective algebra}.
\newblock {\em Proc. London Math. Soc.}, 116(5):1029--1074, 2018.

\bibitem[YZ20]{YZ16}
Y.~Yang and G.~Zhao.
\newblock On two cohomological {H}all algebras.
\newblock {\em Proc. Roy. Soc. Edinburgh Sect. A}, 150(3):1581--1607, 2020.

\end{thebibliography}

%\vfill

%\textsc{\small B. Davison: School of Mathematics, University of Edinburgh, James Clerk Maxwell Building, Peter Guthrie Tait Road, King's Buildings, Edinburgh EH9 3FD, United Kingdom}\\
%\textit{\small E-mail address:} \texttt{\small ben.davison@ed.ac.uk}\\

\end{document}